\documentclass[11pt]{amsart}

\usepackage{amsmath,amssymb,amsthm}
\usepackage{mathtools}
\usepackage[margin=1in]{geometry}
\usepackage{xcolor}
\usepackage[colorlinks=true,linkcolor=blue,citecolor=blue,urlcolor=blue]{hyperref}

\theoremstyle{plain}
\newtheorem{theorem}{Theorem}[section]
\newtheorem{lemma}[theorem]{Lemma}
\newtheorem{proposition}[theorem]{Proposition}
\newtheorem{corollary}[theorem]{Corollary}

\newtheorem{maintheorem}{Theorem}

\theoremstyle{definition}
\newtheorem{definition}[theorem]{Definition}
\newtheorem{example}[theorem]{Example}

\theoremstyle{remark}
\newtheorem{remark}[theorem]{Remark}

\numberwithin{equation}{section}

\newcommand{\R}{\mathbb{R}}
\newcommand{\relu}{\operatorname{ReLU}}   
\newcommand{\N}{\mathcal{N}}   
\newcommand{\C}{\mathcal{C}}   
\newcommand{\F}{\mathcal{F}}   
\newcommand{\wsup}{\mathrm{wsup}}
\DeclareMathOperator{\conv}{conv}

\title[Geometry of ReLU function spaces]{The Geometry of Limits in ReLU Network Function Spaces}

\author{Kathryn A. Lindsey}
\address{Department of Mathematics, Boston College, Chestnut Hill, MA 02467, U.S.A.}
\email{lindseka@bc.edu}


\date{\today}

\begin{document}

\begin{abstract}
We study the geometry and closure of spaces of functions
realized by ReLU neural networks of fixed architecture.
We first establish a general result about limits of
continuous piecewise-linear (CPWL) functions: the minimum
number of convex sets on which a function is affine is
lower semicontinuous under pointwise convergence.
The proof uses a connection to hypergraph coloring and
a compactness principle, and yields quantitative bounds
on the polyhedral complexity of continuous limits.
We then characterize the functions realizable by shallow
ReLU networks with vector-valued outputs, extending a
known scalar-valued characterization. Using this and
the geometry of limiting hyperplane arrangements, we
prove that fixed-width shallow scalar ReLU realization
spaces are closed in the weighted supremum norm on the
whole input space. In contrast, we construct an explicit
sequence of networks with architecture $(2,3,2\mid1)$
whose limit is not realizable by that architecture.
Non-closedness holds both globally in weighted supremum
norm and under uniform convergence on compact sets with
nonempty interior, disproving a conjecture of Petersen,
Raslan, and Voigtlaender.
Finally, we show that non-closedness for a fixed
scalar-output ReLU architecture necessarily involves a
\emph{blow-up residue}, in which diverging output
coefficients amplify an emerging linear dependence among
hidden features. Together, these results provide
geometric and functional-analytic tools for studying
the boundaries and degenerations of ReLU realization
spaces.
\end{abstract}

\maketitle

\section{Introduction}
\label{s:intro}

A central theme in neural network approximation is the tension between
expressive power and architectural constraints. By the universal
approximation theorem, shallow ReLU networks of unrestricted width are
dense in the continuous functions on any compact subset of Euclidean
space \cite{leshno1993nonpolynomial,devorehaninpetrova2021}. Every
scalar-valued continuous piecewise-linear (CPWL) function has an exact
ReLU representation \cite{arora2018understanding}, and recent work of
Brunck, Hertrich, Ruess, Stade, and Yehudayoff
\cite{BrunckEtAl2026} shows that two hidden layers suffice for exact
representation if their widths may be chosen freely. These theorems
concern what can be achieved as the architecture varies. For a
\emph{fixed} architecture, however, the functions obtainable form a
particular finite-parameter family inside a much larger function
space. Universal approximation says little about the geometry,
closure, or boundary of any one such family, or about the relationships
between different families.

This distinction is relevant to optimization. A network is trained
by varying its parameters, while a data-fitting loss depends on the
function they realize. The relationship between parameter space and
function space can be complicated: the realization map is generally
far from injective, and functions that are close in a function-space
norm need not admit correspondingly close parameter representations.
The interplay between expressive power and approximation stability
is surveyed by DeVore, Hanin, and
Petrova~\cite{devorehaninpetrova2021}, while Petersen, Raslan, and
Voigtlaender~\cite{petersen2021topological} study the topology and
inverse stability of the realization map. Its nontrivial fibers
arise in part from neuron permutations and positive rescalings,
but additional hidden symmetries can also occur
\cite{grigsby2023hiddensymmetries}. The fibers and local functional
dimension of the realization map have been investigated in
\cite{grigsby2025functional}. The present work takes a
complementary perspective, studying the image of the realization
map and its closure: which functions can be realized by a given
architecture, and which can arise as limits of such realizations?

Understanding the image of the realization map, and how this image
sits within a space of functions, is therefore a natural part of
understanding the mathematical structure of parameterized learning
models. The present paper studies this geometry, rather than the
dynamics of any particular optimization algorithm.

In particular, if a sequence of functions realized by networks with
the same architecture converges, must its limit be realized by that
architecture? Petersen, Raslan, and Voigtlaender
\cite{petersen2021topological} investigated this question
systematically. They proved closedness for shallow ReLU networks in
the uniform norm on bounded cubes and conjectured that fixed-complexity
ReLU realization spaces are closed in the uniform topology more
generally. We give a negative answer to this conjecture, with a
counterexample involving only two hidden layers. In parallel, we
establish a global closedness theorem for shallow scalar networks and
develop a general theory of the geometry of limits of CPWL functions.

For an architecture $A=(n_0,n_1,\ldots,n_d\mid n_{d+1})$, we take
$n_0$ to be the input dimension, $n_1,\ldots,n_d$ the hidden-layer
widths, and $n_{d+1}$ the output dimension. Unless otherwise indicated,
networks have biases and an affine output layer. Denote its
finite-dimensional parameter space by $\Theta_A$ and its
parameter-to-function realization map by $\rho_A$. We write
$\F_A=\rho_A(\Theta_A)$ for the corresponding set of functions
$\R^{n_0}\to\R^{n_{d+1}}$, and $\F_A^0$ for the bias-free realization space.
For a compact set $\Omega\subseteq\R^{n_0}$, the notation
$\F_A|_\Omega$ denotes the restrictions of these functions to
$\Omega$.

Our principal topology for functions on the whole input space is
induced by the \emph{weighted supremum norm}
$$
 \|f\|_\wsup=\sup_{x\in\R^n}\frac{\|f(x)\|}{1+\|x\|},
$$
on the Banach space $C_\wsup(\R^n,\R^p)$ of continuous functions
with finite such norm. Every ReLU network function lies in this
space, since it has at most linear growth. This topology controls
behavior at infinity without requiring network functions to be
bounded. It is stronger than compact-open convergence: for example,
$\relu(x+m)-m$ converges to $x$ uniformly on compact subsets of
$\R$, though $x\notin\F_{(1,1\mid1)}$, and this convergence does
not hold in weighted supremum norm
(Example~\ref{ex:compact-open}). Weighted global norms have also
been used to study ReLU approximation with variable width, as in
Hatano, Ikeda, Ishikawa, and Sawano
\cite{Hatano}.

We first study what geometric information is retained under
limits of CPWL functions, independently of any network
architecture. The first result is substantially more general and does
not require the functions in the sequence to be continuous. For a
map $f:X\to\R^p$ on a convex set $X\subseteq\R^n$, define its
\emph{convex-affine complexity} $\kappa(f)$ to be the least number
of convex sets covering $X$ on each of which $f$ agrees with an
affine map, and set $\kappa(f)=\infty$ if there is no finite such
cover.

\begin{maintheorem}[Stability of convex-affine complexity]
\label{mt:complexity}
Let $X\subseteq\R^n$ be convex and let $f_m:X\to\R^p$ be
functions converging pointwise to $f:X\to\R^p$. Then
$$
 \kappa(f)\le\liminf_{m\to\infty}\kappa(f_m).
$$
If, in addition, $X=\R^n$, $f$ is continuous, and
$\kappa(f_m)\le k$ for every $m$, then $f$ is CPWL. More
precisely, it is affine on the cells of a polyhedral complex with
at most
$$
 \sum_{j=0}^n\binom{p\binom{k}{2}}{j}
$$
top-dimensional cells.
\end{maintheorem}

Thus bounded convex-affine complexity is stable under the very weak
hypothesis of pointwise convergence. The proof is not a direct
compactness argument for polyhedral decompositions: the convex sets
in a covering may vary arbitrarily with the sequence, and the
limiting function need not initially have any prescribed geometric
structure. Instead, we encode convex-affine coverings as colorings
of a hypergraph and apply a finite-coloring compactness principle
\cite{Rado1949,deBruijnErdos1951} to show that failure of a
$k$-piece covering has a finite witness.
Continuity of the limit is then enough to recover a polyhedral
complex and the stated bound. In particular, any continuous limit
of functions realized by a fixed ReLU architecture is still CPWL,
even when it cannot be realized by that architecture.

We also investigate how the affine pieces themselves behave under
convergence. Affine formulas on surviving full-dimensional cells
converge, and the Jacobian jump across a codimension-one face of
the limit can be recovered as a sum of jumps across faces that
accumulate onto it. This viewpoint is related to the study of
bent hyperplane arrangements and canonical polyhedral complexes
associated to ReLU networks
\cite{grigsby2022transversality,masden2025algorithmic}, but our
emphasis is on the behavior of geometric structures under limits.
These results provide a foundation for the analysis of shallow
networks and make clear that the non-closedness phenomenon below
is not a failure of CPWL structure to persist. More broadly, they
provide geometric constraints on possible limits in a fixed
realization space, even before one asks whether those limits admit
representations with the same architecture.

We next turn to networks with one hidden layer. The non-affine locus of a
realized function is a union of hyperplanes. Each \emph{visible
bend} is a hyperplane across which the Jacobian has a nonzero,
constant jump. The bends and jumps determine the non-affine part
of the function, leaving an affine term whose realization may
require additional neurons. Dereich and Kassing
\cite{dereich2022minimal} characterized the minimum number of
neurons for scalar-valued shallow ReLU functions: with $q$
visible bends, the minimum is one of $q$, $q+1$, or $q+2$. We
extend the characterization to vector-valued functions, for
which the remaining affine term has a more involved rank cost.

To state the result, a \emph{visible bend presentation} of a
continuous $G:\R^n\to\R^p$ consists of affine functions
$L_j(x)=u_j^Tx+b_j$, $1\le j\le q$, with $u_j\ne0$ and
pairwise distinct zero hyperplanes, such that $G$ is affine on
each chamber of their arrangement and has a constant, nonzero
Jacobian jump across each hyperplane. Such a presentation gives
$$
 G(x)=B+Ax+\frac12\sum_{j=1}^q v_j|L_j(x)|,
 \qquad A\in\R^{p\times n},\quad v_j\in\R^p\setminus\{0\}.
$$
For $S\subseteq\{1,\ldots,q\}$ and signs
$\epsilon_j\in\{-1,1\}$ for $j\notin S$, define
$$
 U_S=\operatorname{span}\{u_j:j\in S\},\qquad
 M_{S,\epsilon}=A-\frac12\sum_{j\notin S}\epsilon_jv_ju_j^T,
$$
and let the \emph{affine residual cost} $\tau(G)$ be
$$
 \tau(G)=\min_{S,\epsilon}
 \left\{|S|+2\operatorname{rank}
 \bigl(M_{S,\epsilon}|_{U_S^\perp}\bigr)\right\}.
$$
This counts the neurons needed beyond one for each visible bend;
it is independent of the chosen presentation.

\begin{maintheorem}[Characterization of shallow realizations]
\label{mt:characterization}
Let $G:\R^n\to\R^p$ be continuous and $N\ge1$.
Then $G\in\F_{(n,N\mid p)}$ if and only if $G$ admits a
visible bend presentation with $q$ visible hyperplanes such that
$$
 q+\tau(G)\le N.
$$
\end{maintheorem}

In scalar output dimension, this recovers the characterization
of Dereich and Kassing. For global limits of shallow scalar
networks, the characterization combines with the behavior of
visible hyperplanes and their gradient jumps to give the following.

\begin{maintheorem}[Global closedness of shallow scalar networks]
\label{mt:shallowclosed}
For all $n,N\ge1$, the space $\F_{(n,N\mid1)}$ is closed in
$C_\wsup(\R^n,\R)$.
\end{maintheorem}

This is a global counterpart to the result of Petersen, Raslan,
and Voigtlaender for uniform convergence on bounded cubes
\cite{petersen2021topological}. The two results
address different convergence hypotheses: our theorem concerns
functions on the entire input space with controlled growth at
infinity. In the proof, visible hyperplanes may merge or escape to
infinity, and jumps may cancel. The weighted supremum norm controls the affine
contributions of escaping neurons in a way compact-open
convergence does not. Related existence questions for optimal
shallow ReLU approximations have been studied in
\cite{dereich2024,dereichjentzenkassing2024}.

For networks with two hidden layers, the corresponding closedness
statement is false,
even if one requires convergence globally in weighted supremum
norm. We obtain an explicit example for an architecture with
three neurons in the first hidden layer and two in the second.

\begin{maintheorem}[Non-closedness at depth two]
\label{mt:nonclosed}
Let $A=(2,3,2\mid1)$. Neither $\F_A$ nor its bias-free
realization space $\F_A^0$ is closed in
$C_\wsup(\R^2,\R)$. Furthermore, for every compact set
$\Omega\subseteq\R^2$ with nonempty interior, the
restriction space $\F_A|_\Omega$ is not closed in
$C(\Omega)$ with the uniform norm.
\end{maintheorem}

The compact-domain assertion disproves the conjecture of
Petersen, Raslan, and Voigtlaender
\cite[Section~3.4]{petersen2021topological} for the ReLU activation.
As observed in \cite[Section~3.3]{petersen2021topological},
non-closedness implies that functions in the closure of
$\F_A|_\Omega$ but not in $\F_A|_\Omega$ have no best approximation
in $\F_A|_\Omega$, and that the weights of any approximating sequence
must diverge.

The sequence behind the theorem is simple to specify. Its first
layer computes
$$
 u_1=\relu(x_1),\qquad u_2=\relu(x_2),\qquad
 u_3=\relu(x_1+x_2).
$$
With
$$
 g_0=u_1+u_2-u_3,\qquad
 h_1=-4u_1+u_2+2u_3,\qquad
 h_2=-2u_1+4u_2,
$$
we take
$$
 f_M=\relu(Mg_0+h_1)-\relu(Mg_0+h_2).
$$
The function $g_0$ is nonnegative, but vanishes on two
full-dimensional regions. On these regions, the separate
nonlinearities of the two second-layer neurons survive in the
limit; where $g_0>0$, the large common term cancels. The
weighted-supremum limit is CPWL, but the configuration of its
nonlinear locus and its affine gradients cannot be produced by
the prescribed widths. Thus the obstruction lies not in the
complexity of the limit as a CPWL function, but in the way its
pieces would have to be assembled by the network layers. A
localization argument gives the compact-domain assertion for
networks with biases.

The distinction between the limit's CPWL geometry and the constraints
of its architecture raises questions beyond closedness. How can one
compare the sets $\F_A$ and $\F_B$ for different architectures?
When a function lies in the closure of one realization space but not
in the space itself, can it be realized by a slightly larger network?
What geometric data describe such degenerations? Our example shows
that these questions cannot be answered solely by counting affine
regions.

The counterexample also illustrates a general mechanism for
limits that leave a fixed realization space. By positive
homogeneity of ReLU, hidden-layer parameters can be rescaled
into a bounded set; hence, along a subsequence, the final hidden
features converge in weighted supremum norm. If the final
output coefficients remained bounded, the limit would be
realizable. The only possible obstruction is therefore that
diverging output coefficients magnify an emerging linear
dependence among the limiting hidden features. This is a weighted-sup-norm refinement of the exploding-weights
phenomenon of \cite[Section~3.3]{petersen2021topological}. We call the
surviving finite contribution a \emph{blow-up residue}.

\begin{maintheorem}[Blow-up residues and non-closedness]
\label{mt:residues}
Let $A=(n_0,\ldots,n_d\mid1)$ be a fixed scalar-output
ReLU architecture with $d\ge1$, and suppose that
$f_m\in\F_A$ converges in $C_\wsup(\R^{n_0},\R)$ to
$f\notin\F_A$. After passing to a subsequence, there are
representations of the $f_m$ with final hidden features
$h_{m,j}\to h_j$ in weighted supremum norm, $1\le j\le N=n_d$.
Define linear maps from $\R^{N+1}$ to $C_\wsup(\R^{n_0},\R)$ by
$$
 U_m(a_0,\ldots,a_N)=a_0+\sum_{j=1}^N a_jh_{m,j},\qquad
 U(a_0,\ldots,a_N)=a_0+\sum_{j=1}^N a_jh_j.
$$
Writing $f_m=U_m\alpha_m$, there exist vectors
$\beta_m\to\beta$, unit vectors $\gamma_m\to\gamma\in\ker U$,
and positive scalars $M_m\to\infty$ such that
$$
 \alpha_m=\beta_m+M_m\gamma_m,\qquad
 f=U\beta+R,\qquad
 R=\lim_{m\to\infty}M_mU_m\gamma_m\notin\operatorname{im}U.
$$
\end{maintheorem}

This theorem identifies a necessary mechanism, not a
characterization of non-closedness: a blow-up residue can be
realizable by another choice of hidden features. For the
sequence above, the two normalized second-layer features
converge to the same function $g_0$, while their magnified
difference converges to the nonrealizable limit. The residue
explains how the limit can arise; the geometric obstruction
shows why it cannot be realized by the original architecture.

Taken together, the results suggest a broader study of the family of
ReLU realization spaces and their relationships. Natural problems
include determining when one realization space is contained in
another, describing their closures and boundaries, and identifying
which boundary functions become realizable after enlarging an
architecture. The blow-up residue viewpoint also points toward
compactifications that record degenerating network representations
and the new limiting contributions they produce. We do not construct
such a compactification here, but the results exhibit both the
geometric constraints that persist under limits and a mechanism
by which fixed-architecture realizability can fail. A theory of this
kind could ultimately inform questions about the existence of optimal
parameters and the behavior of loss functions near degenerate
representations, while keeping distinct the geometry of realizable
functions and the dynamics of training.

\subsection{Organization}

Section~\ref{s:background} introduces ReLU networks,
realization spaces, and the polyhedral geometry of CPWL
functions. Section~\ref{s:complexity} develops the theory
of limits of functions with bounded convex-affine
complexity, including the hypergraph-coloring approach
and the behavior of affine pieces and gradient jumps
under convergence. Section~\ref{s:wsup} introduces the
weighted supremum norm and its interpretation through
compactification of the input space.
Section~\ref{s:shallow} characterizes shallow ReLU
realization spaces and establishes closedness for
scalar-valued shallow networks.
Section~\ref{s:nonclosed} constructs the depth-two
counterexample and proves its nonrealizability.
Finally, Section~\ref{s:blowup-residues} develops the
theory of blow-up residues and applies it to the
counterexample.

\subsection*{Acknowledgements}

The author was partially supported by NSF DMS SCALE MoDL
Program Award \#2133822. The author thanks J. Elisenda
Grigsby for many years of fruitful mathematical
collaboration and conversation, including numerous
discussions related to the ideas in this paper. The author
also thanks Chenxi Wu and Yaoying Fu for helpful
discussions.

Throughout the development of this paper, the author made
extensive use of commercially available versions of
OpenAI's ChatGPT and Anthropic's Claude. Over many months,
the author engaged in lengthy, iterative mathematical
conversations with these models, using them to test conjectures, develop and scrutinize
arguments, and refine the exposition. The mathematical
questions and overall direction of the investigation were
determined by the author, with the models serving as
interactive tools in the research process.
In particular, the counterexample in Section~\ref{s:nonclosed} was
found by Claude Opus 5.5, in a search directed by the author: the
author identified the architecture $(2,3,2\mid1)$ as a promising
place to search and explained the blow-up residue mechanism
developed in Section~\ref{s:blowup-residues}, and Claude conducted
an extended search lasting over an hour, finding the counterexample
together with an explanation of why it works. All results and proofs in this paper were checked by the author,
who takes responsibility for its content.

\section{Background and notation}
\label{s:background}

\subsection{ReLU neural networks and realization spaces}
\label{ss:relu-realization}

The \emph{rectified linear unit} is the function
$\relu:\R\to\R$ defined by $\relu(t)=\max\{0,t\}$.
For each $n\ge 1$, we write
$\sigma:\R^n\to\R^n$ for its coordinatewise application.
Throughout the paper, neural networks are fully connected and
feedforward, with ReLU activation in every hidden layer and an
affine output layer. Biases are allowed unless otherwise specified.

\begin{definition}[ReLU networks and their parameters]
\label{def:relu-network}
Fix positive integers $n_0,n_1,\ldots,n_d,n_{d+1}$, where $d\ge 1$.
A \emph{ReLU network architecture} is a tuple
$$
 A=(n_0,n_1,\ldots,n_d\mid n_{d+1}).
$$
Here $n_0$ is the input dimension, $n_1,\ldots,n_d$ are the
widths of the $d$ hidden layers, and $n_{d+1}$ is the output
dimension. A \emph{parameter} for $A$ is a tuple
$$
 \theta=(W^1,b^1,\ldots,W^{d+1},b^{d+1}),
 \qquad
 W^\ell\in\R^{n_\ell\times n_{\ell-1}},\quad
 b^\ell\in\R^{n_\ell}\quad(1\le\ell\le d+1).
$$
The pair $(W^{d+1},b^{d+1})$ gives the output layer; we also write
$W^{\mathrm{out}}=W^{d+1}$ and $b^{\mathrm{out}}=b^{d+1}$.
The \emph{parameter space} $\Theta_A$ is the Euclidean space of all
such tuples, of dimension
$$
 D(A)=\sum_{\ell=1}^{d+1}n_\ell(n_{\ell-1}+1).
$$

Given $\theta\in\Theta_A$, define the hidden-layer maps
$h^\ell_\theta:\R^{n_0}\to\R^{n_\ell}$ recursively by
$$
 h^0_\theta(x)=x,\qquad
 h^\ell_\theta(x)
 =\sigma\!\left(W^\ell h^{\ell-1}_\theta(x)+b^\ell\right),
 \quad 1\le\ell\le d.
$$
The \emph{function realized by} $\theta$ is
$$
 f_\theta(x)=W^{d+1}h^d_\theta(x)
+b^{d+1},\qquad x\in\R^{n_0}.
$$
\end{definition}

\begin{definition}[Realization map and realization spaces]
\label{def:realization}
For a fixed architecture $A=(n_0,\ldots,n_d\mid n_{d+1})$,
the \emph{realization map} is
$$
 \rho_A:\Theta_A\longrightarrow C(\R^{n_0},\R^{n_{d+1}}),
 \qquad \theta\longmapsto f_\theta.
$$
Its image
$$
 \F_A\coloneqq\rho_A(\Theta_A)
   =\{f_\theta:\theta\in\Theta_A\}
$$
is the \emph{realization space} of $A$.

Let $\Theta_A^0\subseteq\Theta_A$ be the subspace of parameters
for which every bias vector, including $b^{d+1}$, is zero.
The \emph{bias-free realization space} is
$$
 \F_A^0\coloneqq\rho_A(\Theta_A^0).
$$
For any $\Omega\subseteq\R^{n_0}$, write
$$
 \F_A|_\Omega
 \coloneqq\{f|_\Omega:f\in\F_A\},\qquad
 \F_A^0|_\Omega
 \coloneqq\{f|_\Omega:f\in\F_A^0\}.
$$
For $f\in\F_A$, the set $\rho_A^{-1}(f)\subseteq\Theta_A$
is called the \emph{fiber} of the realization map over $f$.
\end{definition}

In particular, $(n,N\mid p)$ denotes an architecture with a
single hidden layer of width $N$ and output dimension $p$, while
$(2,3,2\mid1)$ denotes an architecture with two hidden layers of
widths $3$ and $2$ and scalar output.
We regard $\F_A$ as a set of functions, rather than a set of
parameterized representations; different parameters may realize
the same function. The topology on the ambient function space
will be specified in Section~\ref{s:wsup}.

The ReLU activation is positively homogeneous:
$\relu(ct)=c\relu(t)$ for every $c\ge 0$.
Consequently, every bias-free ReLU network realizes
a positively homogeneous function, satisfying
$$
f_\theta(cx)=cf_\theta(x)
\qquad (c\ge 0).
$$
Positive homogeneity also gives rise to parameter
symmetries: multiplying the incoming weights and
bias of a hidden neuron by $c>0$, while dividing
its outgoing weights by $c$, leaves the realized
function unchanged. Thus the realization map is
generally noninjective, even before accounting
for permutations of neurons within a layer.

\subsection{CPWL functions and polyhedral complexes}
\label{ss:cpwl-background}

Classical exact representations of CPWL functions include the
generalized hinging-hyperplane representations of Wang and
Sun~\cite{wangsun2005hinging}, in which the number of terms may depend
on the function.  Here we are interested instead in the geometric
constraints imposed by a fixed neural network architecture, and in
how these constraints behave under limits.

The polyhedral geometry of ReLU networks has been studied from several
perspectives.  Hanin and Rolnick~\cite{haninrolnick2019regions}
investigate the number and boundaries of linear regions, while
Grigsby and Lindsey~\cite{grigsby2022transversality} and
Masden~\cite{masden2025algorithmic} study polyhedral complexes
associated to ReLU networks.

\begin{definition}[Polyhedron and face]
\label{def:polyhedron}
A subset $P\subseteq\R^n$ is a \emph{polyhedron} if it is the
intersection of finitely many closed affine half-spaces.  In particular,
polyhedra need not be bounded or full-dimensional.  A \emph{face} of a
polyhedron is a subset on which a linear functional attains its maximum
(when that maximum is finite); we include $P$ itself among its faces.
\end{definition}

\begin{definition}[Polyhedral complex]
\label{def:polyhedral-complex}
A \emph{polyhedral complex} $\C$ in $\R^n$ is a finite collection of
polyhedra, called \emph{cells}, such that every nonempty face of a cell
is a cell of $\C$ and the intersection of any two cells is either empty
or a common face of both.  Its \emph{underlying space} is
$|\C|=\bigcup_{P\in\C}P$.
A cell of dimension $n$ is \emph{top-dimensional} when $|\C|=\R^n$.
The \emph{$k$-skeleton} $\operatorname{sk}_k(\C)$ is the subcomplex
consisting of cells of dimension at most $k$.
\end{definition}

\begin{definition}[Finitely piecewise linear and CPWL functions]
\label{def:CPWL}
A function $f:\R^n\to\R^m$ is \emph{finitely piecewise linear} if
there are finitely many polyhedra $P_1,\ldots,P_N$ covering $\R^n$
and affine maps $\ell_i:\R^n\to\R^m$ such that
$f|_{P_i}=\ell_i|_{P_i}$ for each $i$.
We say that $f$ is \emph{continuous piecewise linear} (\emph{CPWL})
if it is continuous and finitely piecewise linear.
Here, as throughout, ``piecewise linear'' allows affine, rather than
necessarily homogeneous linear, formulas.
\end{definition}

A CPWL function is affine on the cells of some finite polyhedral
complex covering $\R^n$.  We emphasize that such a complex is not
unique: it may be subdivided without changing the function.  The
geometry of ReLU networks provides a distinguished complex associated
to each \emph{parameterized network}, as follows.

\begin{definition}[Canonical polyhedral complex of a ReLU network]
\label{def:canonical-complex}
Fix an architecture $A=(n_0,n_1,\ldots,n_d\mid n_{d+1})$ and
$\theta\in\Theta_A$.  For each hidden layer $1\le\ell\le d$, let
$\mathcal R_\ell$ be the polyhedral complex of $\R^{n_{\ell-1}}$
induced by the affine hyperplanes
$$
\bigl\{y\in\R^{n_{\ell-1}}:
   (W^\ell y+b^\ell)_j=0\bigr\},
\qquad 1\le j\le n_\ell.
$$
Empty zero sets and equations vanishing identically are discarded.

Define polyhedral complexes $\C_\theta^{(\ell)}$ in the input space
$\R^{n_0}$ inductively.  Start with
$\C_\theta^{(0)}=\{\R^{n_0}\}$, and for $1\le\ell\le d$ set
$$
\C_\theta^{(\ell)}
=\left\{
 P\cap (h_\theta^{\ell-1})^{-1}(Q)
 : P\in\C_\theta^{(\ell-1)},\ Q\in\mathcal R_\ell,
\quad P\cap (h_\theta^{\ell-1})^{-1}(Q)\ne\varnothing
\right\},
$$
where $h^0_\theta$ is the identity and repeated cells are identified.
The \emph{canonical polyhedral complex} of the parameterized network
$\theta$ is
$$
\C_\theta:=\C_\theta^{(d)}.
$$
\end{definition}

This is the iterated pullback construction of Grigsby
and Lindsey~\cite{grigsby2022transversality};
see also Masden~\cite{masden2025algorithmic}. By
\cite{grigsby2022transversality}, each
$\C_\theta^{(\ell)}$ is a finite polyhedral complex
covering $\R^{n_0}$, and $h_\theta^\ell$ is affine
on every cell of $\C_\theta^{(\ell)}$.
Consequently, $f_\theta$ is CPWL and affine on
every cell of $\C_\theta$.

\begin{definition}[Non-affine locus]
\label{def:nonaffine-locus}
For a CPWL function $f:\R^n\to\R^m$, its \emph{non-affine locus} is
$$
\N(f)=\{x\in\R^n:
  f\text{ is not affine on any open neighborhood of }x\}.
$$
\end{definition}

The distinction between the canonical complex $\C_\theta$ and the
non-affine locus $\N(f_\theta)$ is important.  The former records
the hidden-layer subdivisions of a particular parameterized network;
the latter depends only on the realized function.  In particular,
$$
\N(f_\theta)\subseteq |\operatorname{sk}_{n_0-1}(\C_\theta)|,
$$
and the inclusion can be strict, since bends introduced in hidden
layers may cancel at the output.  Thus different parameters realizing
the same function need not have the same canonical polyhedral complex.

\section{Limits of functions with boundedly many affine pieces}
\label{s:complexity}

\subsection{Convex-affine complexity and hypergraph coloring}

\subsubsection{Background on hypergraphs and colorings}

We recall some standard terminology; see, e.g.,
\cite{bretto2013hypergraph}.
A \emph{hypergraph} $\mathcal H=(V,E)$ consists of a set $V$ of
\emph{vertices} together with a collection $E$ of nonempty subsets of
$V$, called \emph{hyperedges}. We allow $V$ to be infinite. A graph is
the special case in which every hyperedge has exactly two elements.
For $S\subseteq V$, the \emph{induced subhypergraph} $\mathcal H[S]$
has vertex set $S$ and hyperedges $\{T\in E: T\subseteq S\}$.

A \emph{$k$-coloring} of $\mathcal H$ is a map
$c:V\to\{1,\ldots,k\}$. A set $T\subseteq V$ is
\emph{monochromatic} under $c$ if $c|_T$ is constant, and $c$ is
\emph{proper} if no hyperedge is monochromatic. We will only
consider hypergraphs whose hyperedges contain at least two
vertices, since a singleton hyperedge can never be properly
colored. The \emph{chromatic number}
$\chi(\mathcal H)\in\mathbb{N}\cup\{\infty\}$ is the least $k$ for
which a proper $k$-coloring exists, and $\chi(\mathcal H)=\infty$ if
there is no such $k$.

The classical compactness theorem of de Bruijn and
Erd\H{o}s~\cite{deBruijnErdos1951} states that an
infinite graph is $k$-colorable if and only if every
finite subgraph is $k$-colorable. The same conclusion
holds for hypergraphs with finite hyperedges and can
be viewed as an instance of Rado's selection
principle~\cite{Rado1949}; see also
Truszczy\'nski and Tuza~\cite{truszczynski1992}.
Since we use this compactness principle essentially,
we include a short proof using Tychonoff's theorem.

\begin{theorem}[compactness of colorings]
\label{t:coloring-compactness}
Let $\mathcal H=(V,E)$ be a hypergraph all of whose hyperedges are
finite, and let $k\in\mathbb{N}$. If $\mathcal H[S]$ admits a proper
$k$-coloring for every finite $S\subseteq V$, then $\mathcal H$
admits a proper $k$-coloring. Consequently,
\[
\chi(\mathcal H)=\sup_{\substack{S\subseteq V\\ S\text{ finite}}}\chi(\mathcal H[S]).
\]
\end{theorem}

\begin{proof}
Give $\{1,\ldots,k\}$ the discrete topology, so that
$K=\{1,\ldots,k\}^V$ is compact by Tychonoff's theorem. For finite
$S\subseteq V$, let $G_S\subseteq K$ be the set of colorings whose
restriction to $S$ is a proper coloring of $\mathcal H[S]$. Each
$G_S$ is nonempty by hypothesis, and closed, since membership depends
only on the finitely many coordinates in $S$. Since
$G_{S_1\cup\cdots\cup S_r}\subseteq G_{S_1}\cap\cdots\cap G_{S_r}$, the
family has the finite intersection property, so
$\bigcap_SG_S\neq\varnothing$. Any $c$ in the intersection is proper:
each hyperedge $T$ is finite, so $c\in G_T$.
\end{proof}

\subsubsection{The convex-affine obstruction hypergraph}

To apply Theorem~\ref{t:coloring-compactness} to the question of how
many convex-affine pieces a function requires, we encode this question
as a hypergraph coloring problem. We formulate the following two definitions.

\begin{definition}[$f$-convex affine set]
Let $X\subseteq\R^n$ be convex and let $f:X\to\R^m$. We say that $f$ is
\emph{affine on} $C\subseteq X$ if there is an affine map
$\ell:\R^n\to\R^m$ with $f|_C=\ell|_C$.  We say that a set $S\subseteq X$
is \emph{$f$-convex-affine} if $f$ is affine on $\conv(S)$. 
\end{definition}

\begin{definition}[convex-affine obstruction hypergraph $ H_f$]
\label{def:Hf}
Let $X\subseteq\R^n$ be convex and let $f:X\to\R^m$.
Define the \emph{convex-affine obstruction hypergraph} to be the hypergraph 
$\mathcal H_f$ whose  vertex set is $X$ and whose
hyperedges are the finite sets $T\subseteq X$ that are not
$f$-convex-affine.
\end{definition}

Note that, by definition, the hyperedges of $\mathcal H_f$ are exactly the finite
sets of points that are not $f$-convex-affine. We think of a coloring of
$X$ as a candidate assignment of points to convex pieces on which $f$
is affine, with each color class one piece. For this to be possible,
no hyperedge may be monochromatic, since a monochromatic hyperedge would
place a finite set on which $f$ is not affine inside a single piece.

\subsubsection{Convex-affine complexity}

\begin{definition}[convex-affine complexity]
\label{def:kappa}
Let $X\subseteq\R^n$ be convex and $f:X\to\R^m$.  Define the \emph{convex-affine
complexity} $\kappa(f)$ to be the least $k$ such that
$X=\bigcup_{j=1}^kC_j$ with each $C_j$ convex and $f$ affine on each
$C_j$; set $\kappa(f)=\infty$ if no such finite cover exists.
\end{definition}

\begin{proposition}
\label{prop:convex-affine-coloring}
Let $X\subseteq\R^n$ be convex and $f:X\to\R^m$. Then
\[
\kappa(f)=\chi(\mathcal H_f)
=\sup_{\substack{S\subseteq X\\S\text{ finite}}}\chi(\mathcal H_f[S]).
\]
\end{proposition}

\begin{proof}
We first show $\chi(\mathcal H_f)\le\kappa(f)$. If $\kappa(f)=\infty$
there is nothing to prove, so suppose $X=C_1\cup\cdots\cup C_k$ with
each $C_j$ convex and $f$ affine on $C_j$. Define a $k$-coloring
$c:X\to\{1,\ldots,k\}$ by letting $c(x)$ be the smallest index $j$ with
$x\in C_j$. Then every point of color $j$ lies in $C_j$. Now let
$T\subseteq X$ be a finite set on which $c$ is constant, say with
value $j$. Then $T\subseteq C_j$, and since $C_j$ is convex,
$\conv(T)\subseteq C_j$. As $f$ is affine on $C_j$, it is affine on
$\conv(T)$, so $T$ is $f$-convex-affine and is not a hyperedge of
$\mathcal H_f$. Thus no hyperedge is monochromatic, $c$ is a proper
$k$-coloring, and $\chi(\mathcal H_f)\le k$. Taking $k=\kappa(f)$ gives
$\chi(\mathcal H_f)\le\kappa(f)$. 

Conversely, we show $\kappa(f)\le\chi(\mathcal H_f)$. If
$\chi(\mathcal H_f)=\infty$ there is nothing to prove, so let
$c:X\to\{1,\ldots,k\}$ be a proper $k$-coloring and set
$S_j=c^{-1}(j)$. Any finite $F\subseteq S_j$ is monochromatic, hence not
a hyperedge, so $f$ is affine on $\conv(F)$. We claim that $f$ is
affine on $C_j\coloneqq\conv(S_j)$.

We may assume $S_j\neq\varnothing$. Let $A_j$ be the affine hull of
$S_j$ and $d=\dim A_j$. Since $A_j$ is spanned by points of $S_j$, we
may choose affinely independent points $b_0,\ldots,b_d\in S_j$. The set
$B=\{b_0,\ldots,b_d\}$ is finite, so there is an affine map
$\ell:\R^n\to\R^m$ with $f=\ell$ on $\conv(B)$. Now let $x\in C_j$.
Every point of a convex hull is a finite convex combination, so
$x\in\conv(F_x)$ for some finite $F_x\subseteq S_j$. The set
$B\cup F_x$ is finite, so there is an affine map $\ell_x$ with
$f=\ell_x$ on $\conv(B\cup F_x)$. Then $\ell_x$ and $\ell$ agree at the
affinely independent points $b_0,\ldots,b_d$, hence on their affine
hull $A_j$. Since $x\in A_j$, we get $f(x)=\ell_x(x)=\ell(x)$. Thus
$f=\ell$ on $C_j$, proving the claim.

The sets $C_1,\ldots,C_k$ are convex and cover $X$, since
$S_j\subseteq C_j$ and the $S_j$ cover $X$. Hence $\kappa(f)\le k$.
Taking $k=\chi(\mathcal H_f)$ gives $\kappa(f)\le\chi(\mathcal H_f)$,
and therefore $\kappa(f)=\chi(\mathcal H_f)$.

Finally, every hyperedge of $\mathcal H_f$ is finite, so
Theorem~\ref{t:coloring-compactness} gives
$\chi(\mathcal H_f)=\sup_{S}\chi(\mathcal H_f[S])$, the supremum taken
over finite $S\subseteq X$.
\end{proof}

\begin{definition}[$f$-convex-affine-compatible coloring]
\label{def:compatible-coloring}
Let $X\subseteq\R^n$ be convex, $f:X\to\R^m$, $S\subseteq X$, and
$k\in\mathbb N$. A $k$-coloring $c:S\to\{1,\ldots,k\}$ is
\emph{$f$-convex-affine-compatible} if each color class
$c^{-1}(j)$, $1\le j\le k$, is $f$-convex-affine.
\end{definition}

\begin{proposition}[existence of a finite obstruction]
\label{prop:finite-obstruction}
Let $X\subseteq\R^n$ be convex, $f:X\to\R^m$, and $k\in\mathbb N$.
Then $\kappa(f)\le k$ if and only if every finite set
$S\subseteq X$ admits an $f$-convex-affine-compatible $k$-coloring.
Equivalently, $\kappa(f)>k$ if and only if some finite
$S\subseteq X$ admits no such coloring.
\end{proposition}

\begin{proof}
Let $S\subseteq X$ be finite. A $k$-coloring of $S$ is
$f$-convex-affine-compatible if and only if it is a proper coloring
of $\mathcal H_f[S]$. Indeed, if each color class is
$f$-convex-affine and $T\subseteq S$ is monochromatic, then $T$ lies
in a single class $c^{-1}(j)$, so
$\conv(T)\subseteq\conv(c^{-1}(j))$ and $f$ is affine on
$\conv(T)$; thus $T$ is not a hyperedge. Conversely, if $c$ is
proper, each color class is a finite monochromatic set, hence not a
hyperedge, hence $f$-convex-affine. Therefore the condition in the
proposition says $\chi(\mathcal H_f[S])\le k$ for every finite $S$,
which by Proposition~\ref{prop:convex-affine-coloring} is equivalent
to $\kappa(f)\le k$.
\end{proof}

\subsection{Stability under pointwise convergence}

\begin{theorem}[closedness of bounded complexity]
\label{t:complexity-closed}
Let $X\subseteq\R^n$ be convex. For each $k\in\mathbb N$, the set
$\{f:X\to\R^m:\kappa(f)\le k\}$ is closed in the product topology on
$(\R^m)^X$, that is, in the topology of pointwise convergence.
\end{theorem}

\begin{proof}
By Proposition~\ref{prop:finite-obstruction}, the set $\{f:X \to \mathbb{R}^m : \kappa(f) \leq k\}$ is precisely the set 
$$\bigcap_{\textrm{finite }S \subset X} \{f: S \textrm{ admits an $f$-convex-affine-compatible $k$-coloring}\}.$$
We can expand this as 

\begin{equation} \label{eq:expandedseteq}
\bigcap_{\textrm{finite }S \subset X} \left(
\bigcup_{\tau \textrm{ a $k$-coloring of } S}
\bigcap_{a=1}^k\{f: f\text{ is affine on }\conv(S \cap \tau^{-1}(a))\} \right)
\end{equation}

We claim that each of the sets $\{f: f\text{ is affine on } \conv(S \cap \tau^{-1}(a))\}$ is closed. 
Indeed, write $C = \conv(S \cap \tau^{-1}(a)) \subset X$, and choose affinely independent points $b_0,\ldots,b_d\in C$ whose affine
hull equals that of $C$. Each $x$ in this affine hull can be written
uniquely as $x=\sum_i\mu_i(x)\,b_i$ for some vector of coefficients $\mu_i(x) \in \mathbb{R}$ with $\sum_i\mu_i(x)=1$. Then $f$ is affine on
$C$ if and only if $$f(x)=\sum_i\mu_i(x)\,f(b_i)$$ for every $x\in C$.
Each such equation involves the value of $f$ at only finitely many points, and hence the set of $f$s that satisfy the equation is a closed set. The set of functions $f$ that are affine on $C$ is thus the intersection of all of these closed sets, and hence closed, establishing the claim. 

Consequently, in the expression \eqref{eq:expandedseteq}, the parenthesis is a finite union of closed sets.  The conclusion follows.  
\end{proof}

\begin{corollary}[lower semicontinuity of convex-affine complexity]
\label{cor:kappa-lsc}
If $f_j\to f$ pointwise on $X$, then
$\kappa(f)\le\liminf_{j\to\infty}\kappa(f_j)$.
\end{corollary}

\begin{proof}
If the lower limit $k$ is finite, then $\kappa(f_j)\le k$ along a
subsequence, and $f$ lies in the closed set $\{\kappa\le k\}$ by
Theorem~\ref{t:complexity-closed}.
\end{proof}
\begin{remark}
\label{r:no-continuity-needed}
Theorem~\ref{t:complexity-closed} requires neither continuity of the
$f_j$ or of $f$ nor any particular topology beyond pointwise
convergence; in particular it applies to convergence uniformly on
compact sets and in the weighted sup norm. Continuity of the limit is
needed only in the next subsection, to conclude that the limit is
CPWL. 
\end{remark}

\subsection{CPWL structure of the limit}

\begin{lemma}
\label{l:kappa-polyhedral}
Let $f:\R^n\to\R^m$ be continuous with $\kappa(f)\le k$. Then there is
a polyhedral complex $\mathcal P$ covering $\R^n$ such that $f$ is
affine on every cell of $\mathcal P$, and $\mathcal P$ has at most
$\sum_{i=0}^n\binom{m\binom k2}{i}$ top-dimensional cells.
\end{lemma}

\begin{proof}
Choose convex sets $C_1,\ldots,C_r$, $r\le k$, covering $\R^n$ and
affine maps $L_1,\ldots,L_r:\R^n\to\R^m$ with $f|_{C_i}=L_i|_{C_i}$.
Then $f(x)\in\{L_1(x),\ldots,L_r(x)\}$ for every $x$.

For each pair $i<j$ and each coordinate $a\in\{1,\ldots,m\}$ such that
$L_i^a-L_j^a$ is nonconstant, let
$H_{ij,a}=\{x:L_i^a(x)=L_j^a(x)\}$, an affine hyperplane. This gives
an arrangement $\mathcal H$ of $N\le m\binom r2$ hyperplanes; let
$\mathcal P$ be the induced polyhedral complex.

Fix a cell $P\in\mathcal P$ with relative interior $P^\circ$. For each
$i,j$, either $L_i=L_j$ on $P^\circ$, or $L_i(x)\ne L_j(x)$ for every
$x\in P^\circ$. Call $i\sim j$ if $L_i=L_j$ on $P^\circ$, and for each
equivalence class $[i]$ put
$A_{[i]}=\{x\in P^\circ:f(x)=L_i(x)\}$. These sets are pairwise
disjoint, cover $P^\circ$, and are closed in $P^\circ$ by continuity of
$f$. As there are finitely many, each is also open in $P^\circ$. Since
$P^\circ$ is convex, hence connected, exactly one is nonempty, so
$f=L_i$ on $P^\circ$ for some $i$, and by continuity on $P$.

An arrangement of $N$ affine hyperplanes in $\R^n$ has at most
$\sum_{i=0}^n\binom Ni$ chambers~\cite{zaslavsky1975}, and
$N\le m\binom k2$.
\end{proof}

\begin{corollary}[CPWL limits]
\label{cor:cpwl-limits}
Let $f_j:\R^n\to\R^m$ converge pointwise to a continuous function $f$,
and suppose $\kappa(f_j)\le k$ for all $j$. Then $f$ is CPWL and is
affine on the cells of a polyhedral complex with at most
$\sum_{i=0}^n\binom{m\binom k2}{i}$ top-dimensional cells. 
\end{corollary}

\begin{proof}
Combine Theorem~\ref{t:complexity-closed} and
Lemma~\ref{l:kappa-polyhedral}.
\end{proof}

\begin{proof}[Proof of Theorem~\ref{mt:complexity}]
Corollaries~\ref{cor:kappa-lsc} and~\ref{cor:cpwl-limits}
together prove Theorem~\ref{mt:complexity}.
\end{proof}

\begin{corollary}[limits of networks]
\label{cor:network-limits}
Let $A=(n_0,n_1,\ldots,n_d\mid n_{d+1})$ and $s=n_1+\cdots+n_d$.
Every $f\in\F_A$ satisfies $\kappa(f)\le2^{s}$. Consequently,
every limit of a sequence in $\F_A$, uniformly on compact sets or
in the weighted sup norm, is CPWL and is affine on the cells of a
polyhedral complex with at most
$\sum_{i=0}^{n_0}\binom{n_{d+1}\binom{2^s}{2}}{i}$
top-dimensional cells.
\end{corollary}

\begin{proof}
Fix $\theta\in\Theta_A$. To each $x\in\R^{n_0}$ assign its
binary activation pattern in $\{0,1\}^s$, whose entry for neuron $j$ of
layer $\ell$ is $1$ if $(W^\ell h^{\ell-1}_\theta(x)+b^\ell)_j>0$
and $0$ otherwise. Fix a pattern $\pi$ and, for $0\le\ell\le d$,
let $R_\ell$ be the set of $x$ whose pattern agrees with $\pi$ on
layers $1,\ldots,\ell$. We show by induction that $R_\ell$ is
convex and $h^\ell_\theta$ agrees with an affine map $T_\ell$ on
$R_\ell$. This is clear for $\ell=0$. Given the claim for
$\ell-1$, the set $R_\ell$ is the intersection of $R_{\ell-1}$
with the sets where each coordinate of the affine map
$W^\ell T_{\ell-1}+b^\ell$ is $>0$ or $\le0$, as prescribed by
$\pi$; each such set is a half-space, $\R^{n_0}$, or empty, so
$R_\ell$ is convex. On $R_\ell$, $h^\ell_\theta$ is obtained from
$W^\ell T_{\ell-1}+b^\ell$ by zeroing the coordinates where $\pi$
is $0$, which is affine. Hence $f_\theta$ is affine on each of the
at most $2^s$ convex sets $R_d$, which cover $\R^{n_0}$, and
$\kappa(f_\theta)\le2^s$.

Convergence in the weighted sup norm implies uniform convergence
on compact sets, which implies pointwise convergence to a
continuous limit. Corollary~\ref{cor:cpwl-limits}, with
$m=n_{d+1}$ and $k=2^s$, gives the rest.
\end{proof}

\subsection{Pieces and gradient jumps in the limit}
\label{ss:limit-dissections}

\subsubsection{Convex dissections}

We introduce the term \emph{convex dissection} here, rather than working
with polyhedral complexes, because the limit of a sequence of
polyhedral complexes need not be face-to-face. Theorem~\ref{t:limit-dissections} produces a dissection of this
more general kind.

\begin{definition}[convex dissection]
\label{def:dissection}
A \emph{convex dissection} of $\R^n$ is a finite collection
$\{C_1,\ldots,C_N\}$ of closed convex sets with nonempty interiors,
whose interiors are pairwise disjoint and whose union is $\R^n$. The
$C_\alpha$ are its \emph{pieces}. A function is \emph{affine on the
dissection} if it is affine on each piece.
\end{definition}

The top-dimensional cells of a polyhedral complex covering $\R^n$ form
a convex dissection. Every convex dissection is a convex cover, so a
function affine on a dissection with $N$ pieces has $\kappa\le N$.

\begin{definition}[Jacobian and gradient jumps]
Let $f:\R^n\to\R^p$ be continuous and affine on
each piece of a convex dissection $\mathcal D$.
For each piece $C\in\mathcal D$, write $Df|_C$
for the linear part of the unique affine map
agreeing with $f$ on $C$.
Suppose $C,C'\in\mathcal D$ meet in a set $P$
of dimension $n-1$. The \emph{Jacobian jump}
across $P$, cooriented from $C$ to $C'$, is
\[
\delta_f(P)=Df|_{C'}-Df|_C.
\]
When $p=1$, we call this the \emph{gradient jump}.
\end{definition}

\begin{lemma}[Jumps are normal to the wall]
\label{l:jump-normal}
Let $f:\R^n\to\R^p$ be continuous and affine on the
cells of a polyhedral complex covering $\R^n$.
Let $P$ be a codimension-one cell with adjacent
top-dimensional cells $C^+$ and $C^-$, and suppose
$P$ is contained in the hyperplane
$$
H=L^{-1}(0), \qquad L(x)=u\cdot x+\beta,
\quad u\ne 0,
$$
where $C^+$ and $C^-$ lie on the positive and
negative sides of $H$, respectively.

Then the Jacobian jump across $P$ has the form
$$
\delta_f(P)=v u^\top
$$
for some $v\in\R^p$. In particular, when $p=1$,
the gradient jump is parallel to $u$.
\end{lemma}

\begin{proof}
Let $\ell^+$ and $\ell^-$ be the affine formulas
for $f$ on $C^+$ and $C^-$, respectively.
By continuity, $\ell^+-\ell^-=0$ on $P$.
Since $P$ has dimension $n-1$, its relative
interior is an open subset of $H$. Thus the
affine map $\ell^+-\ell^-$ vanishes identically
on $H$, and consequently has the form
$$
\ell^+(x)-\ell^-(x)=vL(x)
$$
for some $v\in\R^p$.
Taking Jacobians gives
$$
\delta_f(P)
=
D\ell^+-D\ell^-
=
v u^\top.
$$
\end{proof}

\begin{lemma}[CPWL functions are determined by their
jumps modulo affine maps]
\label{l:determined-by-jumps}
Let $f,g:\R^n\to\R^p$ be continuous and affine
on every cell of a common finite polyhedral
complex $\mathcal P$ covering $\R^n$.
Then $f-g$ is affine if and only if $f$ and $g$
have the same Jacobian jump across every
codimension-one cell of $\mathcal P$.
\end{lemma}

\begin{proof}
The forward implication is immediate.

Conversely, suppose the jumps agree. Then
$h=f-g$ has zero Jacobian jump across every
codimension-one cell of $\mathcal P$.
Consequently, the linear parts of $h$ agree
on any two top-dimensional cells sharing a
codimension-one face.

Any two top-dimensional cells can be connected
by a sequence of top-dimensional cells in which
consecutive cells share a codimension-one face.
It follows that the linear part of $h$ is the
same on every top-dimensional cell. Since $h$
is continuous, the constant terms also agree
across adjacent cells, and hence throughout
$\R^n$. Therefore $h$ is affine.
\end{proof}

\subsubsection{Convergence of sets}

We use the standard notion of convergence for sets from
\cite{rockafellar-wets}. Write $d(x,A)=\inf_{a\in A}\|x-a\|$,
with $d(x,\varnothing)=\infty$.

\begin{definition}
\label{def:PK}
Let $A_1,A_2,\ldots$ and $A$ be closed subsets of $\R^n$. We say $A_m$
\emph{converges} to $A$, and write $A_m\to A$, if for every $\rho>0$,
\[
\sup_{\substack{a\in A\\ \|a\|\le\rho}}d(a,A_m)\longrightarrow0
\qquad\text{and}\qquad
\sup_{\substack{a\in A_m\\ \|a\|\le\rho}}d(a,A)\longrightarrow0
\qquad\text{as }m\to\infty,
\]
with the convention that a supremum over the empty set is $0$.
\end{definition}

This is the standard Painlev\'e--Kuratowski convergence of sets
\cite{rockafellar-wets}. It adapts convergence in the
Hausdorff metric to potentially unbounded sets. A limit may be empty,
which means that $A_m$ eventually avoids every bounded region, or of
lower dimension: $[m,\infty)\to\varnothing$ and
$[-1/m,1/m]\to\{0\}$. We need two standard facts.

\begin{lemma}
\label{l:PK-facts}
\leavevmode
\begin{enumerate}
\item Every sequence of nonempty closed subsets of $\R^n$ has a
subsequence converging to a closed set, which may be empty
\textup{\cite{rockafellar-wets}}.
\item If $A_m$ is convex for every $m$ and $A_m\to A$, then $A$ is
convex, and every compact set $K\subseteq\operatorname{int}A$ is
contained in $A_m$ for all sufficiently large $m$.
\end{enumerate}
\end{lemma}

\begin{proof}[Proof of \textup{(2)}]
We use Definition~\ref{def:PK} in the following form: if
$a\in A$, then $d(a,A_m)\to0$, so there are $a_m\in A_m$ with
$a_m\to a$; and if $c_m\in A_m$ with $c_m\to c$, then
$d(c_m,A)\to0$, so $c\in A$ because $A$ is closed.

\emph{Convexity.} Let $a,b\in A$ and $t\in[0,1]$. Choose
$a_m,b_m\in A_m$ with $a_m\to a$ and $b_m\to b$. Then
$ta_m+(1-t)b_m\in A_m$ converges to $ta+(1-t)b$, which therefore
lies in $A$.

\emph{Interior.} Let $K\subseteq\operatorname{int}A$ be compact,
and choose $\varepsilon>0$ with $K+\varepsilon\overline B\subseteq A$,
where $\overline B$ is the closed unit ball. Suppose, for a
contradiction, that $K\not\subseteq A_m$ for infinitely many $m$.
Passing to a subsequence, choose $x_m\in K\setminus A_m$. Since
$A\ne\varnothing$, the first condition in Definition~\ref{def:PK}
forces $A_m\ne\varnothing$ for large $m$. As $A_m$ is closed and
convex, there is a unit vector $u_m$ with
$u_m\cdot a\le u_m\cdot x_m$ for all $a\in A_m$. Passing to a
further subsequence, $x_m\to x\in K$ and $u_m\to u$ with $\|u\|=1$.
The point $y=x+\varepsilon u$ lies in $A$, so there are
$a_m\in A_m$ with $a_m\to y$. Letting $m\to\infty$ in
$u_m\cdot a_m\le u_m\cdot x_m$ gives $u\cdot y\le u\cdot x$, that is,
$\varepsilon\le0$, a contradiction.
\end{proof}

We also need the following elementary fact.

\begin{lemma}[pointwise limits of affine maps]
\label{l:affine-limits}
Let $U\subseteq\R^n$ be nonempty and open, and let $g_m$ be affine on
$U$ with $g_m\to g$ pointwise on $U$. Then $g$ is affine on $U$, and
$Dg_m|_U\to Dg|_U$.
\end{lemma}

\begin{proof}
Since $U$ is open, each $g_m$ agrees on $U$ with a unique affine
map $x\mapsto A_mx+c_m$. Fix $x_0\in U$ and $t>0$ such that
$x_i=x_0+te_i\in U$ for $1\le i\le n$. Then
$$
A_me_i=\frac{g_m(x_i)-g_m(x_0)}{t}
\longrightarrow\frac{g(x_i)-g(x_0)}{t},
$$
so $A_m$ converges to the matrix $A$ whose $i$th column is this
limit, and $c_m=g_m(x_0)-A_mx_0\to g(x_0)-Ax_0=:c$. Hence, for
every $x\in U$,
$$
g(x)=\lim_{m\to\infty}(A_mx+c_m)=Ax+c,
$$
so $g$ is affine on $U$ with $Dg|_U=A=\lim_m Dg_m|_U$.
\end{proof}

\subsubsection{Limits of pieces}

\begin{theorem}[limits of affine dissections]
\label{t:limit-dissections}
Let $f_m:\R^n\to\R^p$ converge pointwise to a continuous function $f$,
and suppose each $f_m$ is affine on a convex dissection of $\R^n$ with
at most $N$ pieces. Then, after passing to a subsequence, these
dissections have the same number of pieces
$C^m_1,\ldots,C^m_{N'}$, each $C^m_\alpha$ converges to a closed set
$D_\alpha$, and:
\begin{enumerate}
\item the sets $D_\alpha$ with nonempty interior form a convex
dissection of $\R^n$;
\item $f$ is affine on this dissection;
\item for each $D_\alpha$ with nonempty interior,
$Df_m|_{C^m_\alpha}\to Df|_{D_\alpha}$.
\end{enumerate}
\end{theorem}

\begin{proof}
Pass to a subsequence on which the number of pieces is a constant
$N'\le N$. Pieces are nonempty and closed, so applying
Lemma~\ref{l:PK-facts}(1) to each $\alpha$ in turn, we may pass to a
further subsequence on which $C^m_\alpha\to D_\alpha$ for every
$\alpha$, where each $D_\alpha$ is closed. Each $D_\alpha$ is convex by
Lemma~\ref{l:PK-facts}(2).

\emph{Disjoint interiors.} If $\alpha\ne\beta$ and
$\operatorname{int}D_\alpha\cap\operatorname{int}D_\beta$ contains a
closed ball $K$, then $K\subseteq C^m_\alpha\cap C^m_\beta$ for large
$m$ by Lemma~\ref{l:PK-facts}(2), contradicting the disjointness of
the interiors of $C^m_\alpha$ and $C^m_\beta$.

\emph{Covering.} Let $x\in\R^n$. For each $m$, $x$ lies in some piece
$C^m_\alpha$, so some $\alpha$ occurs for infinitely many $m$. For
those $m$, taking $\rho=\|x\|$ in Definition~\ref{def:PK} gives
$d(x,D_\alpha)\le\sup_{a\in C^m_\alpha,\,\|a\|\le\rho}d(a,D_\alpha)$,
which tends to $0$. Hence $d(x,D_\alpha)=0$, and $x\in D_\alpha$ since
$D_\alpha$ is closed. Thus the $D_\alpha$ cover $\R^n$. A convex set
with empty interior lies in a hyperplane, so the union of the
$D_\alpha$ with empty interior is a closed set with empty interior. Its
complement is therefore dense, and it is contained in the union of the
$D_\alpha$ with nonempty interior, which is closed. Hence the latter
already cover $\R^n$. This proves (1).

\emph{Affineness.} Fix $D_\alpha$ with nonempty interior, and let $B$
be an open ball with $\overline B\subseteq\operatorname{int}D_\alpha$.
By Lemma~\ref{l:PK-facts}(2), $B\subseteq C^m_\alpha$ for large $m$,
so $f_m$ is affine on $B$ with linear part $Df_m|_{C^m_\alpha}$. By
Lemma~\ref{l:affine-limits}, $f$ is affine on $B$, and its linear part
there is $\lim_m Df_m|_{C^m_\alpha}$. Any two such balls are joined by
a chain of overlapping balls in the connected open set
$\operatorname{int}D_\alpha$, and affine maps agreeing on a nonempty
open set are equal. Hence $f$ agrees with a single affine map on
$\operatorname{int}D_\alpha$, and by continuity on its closure
$D_\alpha$. This proves (2) and (3).
\end{proof}
\subsubsection{Gradient jumps}

\begin{corollary}[gradient jumps in the limit]
\label{cor:limit-jumps}
In the setting of Theorem~\ref{t:limit-dissections}, let $D_\alpha$
and $D_\beta$ have nonempty interior, and suppose
$P=D_\alpha\cap D_\beta$ has dimension $n-1$. Then, cooriented from
$D_\alpha$ to $D_\beta$,
\[
\delta_f(P)=\lim_{m\to\infty}
\Bigl(Df_m|_{C^m_\beta}-Df_m|_{C^m_\alpha}\Bigr).
\]
If this limit is nonzero, then $P\subseteq\N(f)$.
\end{corollary}

\begin{proof}
The formula follows from Theorem~\ref{t:limit-dissections}\textup{(3)}
by subtraction. If $\delta_f(P)\ne0$, every point of the relative
interior of $P$ has neighborhoods meeting both
$\operatorname{int}D_\alpha$ and $\operatorname{int}D_\beta$, where
$f$ agrees with affine maps having different linear parts; so $f$ is
not affine near such a point. As $\N(f)$ is closed, $P\subseteq\N(f)$.
\end{proof}

The approximating dissections may contain pieces between
$C^m_\alpha$ and $C^m_\beta$ that collapse or disappear in the limit.
Their contributions to the change in gradient telescope, so only the
two surviving neighbors appear in the formula.

\begin{remark}[what can happen to pieces and bends]
\label{r:piece-fates}
A piece may survive, collapse, or escape. For $f_m=\relu(x-m)$ on $\R$,
with pieces $(-\infty,m]$ and $[m,\infty)$, the limits are $\R$ and
$\emptyset$. If $f_m$ has the three pieces $(-\infty,0]$, $[0,1/m]$,
$[1/m,\infty)$, the middle piece collapses to $\{0\}$, and the jump of
the limit at $0$ is the sum of the jumps of $f_m$ at $0$ and $1/m$.
Bends may also vanish: for $f_m=\tfrac1m\relu(x)$, both pieces survive,
but the jump tends to $0$ and $f=0$. Grouping adjacent pieces on which
$f$ has the same formula gives a coarser decomposition of $\R^n$; in
this sense, the affine structure of the limit is a coarsening of the
limit of the affine structures. In particular,
\[
\N(f)\subseteq\bigcup_{\alpha}\partial D_\alpha.
\]
\end{remark}

\section{Weighted sup norm}\label{s:wsup}
The weighted sup norm provides a natural topology for studying
limits of ReLU network functions on all of $\R^n$, since it
controls their behavior at infinity as well as on compact sets.
In particular, this is the topology in which we will establish
closedness of shallow scalar realization spaces in
Section~\ref{s:shallow}.

We also establish continuity of the realization map in this
topology and describe an interpretation in terms of
compactification. Although these results are not needed for
the proofs that follow, they help explain why the weighted
sup norm is a natural choice for studying ReLU realization spaces.

\subsection{Definition of the weighted sup norm}
The class of continuous functions from $\mathbb{R}^n$ to $\mathbb{R}^m$ with \emph{at most linear growth} is $$C_{\textrm{wsup}}(\mathbb{R}^n,\mathbb{R}^m) \coloneqq \left\{ f:\mathbb{R}^n \to \mathbb{R}^m : f \textrm{ is continuous}, \sup_{x \in \mathbb{R}^n } \frac{||f(x)||_{\mathbb{R}^m}}{1+||x||_{\mathbb{R}^n} } < \infty \right\}.$$ 

\begin{definition} [weighted sup norm] \label{def:wsupnorm}
The \emph{weighted sup norm} on $C_{\textrm{wsup}}(\mathbb{R}^n,\mathbb{R}^m)$  is 
$$||f||_{\textrm{wsup}} \coloneqq  \sup_{x \in \mathbb{R}^n } \frac{||f(x)||_{\mathbb{R}^m}}{1+||x||_{\mathbb{R}^n} }.$$
\end{definition}

It is well-known that $C_{\textrm{wsup}}(\mathbb{R}^n,\mathbb{R}^m)$ equipped with the weighted sup norm is a Banach space. We call its topology the \emph{weighted sup norm topology}. The induced metric (the \emph{weighted sup norm metric}) is 
$$d_{\textrm{wsup}}(f,g) \coloneqq \sup_{x \in \mathbb{R}^n }\frac{||f(x)-g(x)||_{\mathbb{R}^m}}{1+ ||x||_{\mathbb{R}^n}}.$$

\begin{remark}
Convergence in the weighted sup norm is strictly stronger than
uniform convergence on compact sets (convergence in the
compact-open topology). Indeed, if $d_\wsup(f_k,f)\to0$ and
$K\subseteq\{x:\|x\|\le\rho\}$ is compact, then
$\sup_{x\in K}\|f_k(x)-f(x)\|\le(1+\rho)\,d_\wsup(f_k,f)\to0$.
Example~\ref{ex:compact-open} shows that the converse fails.
\end{remark}

\begin{example}
\label{ex:compact-open}
For $m\in\mathbb N$ let $f_m(x)=\relu(x+m)-m$, so that $f_m(x)=x$
for $x\ge-m$ and $f_m(x)=-m$ for $x<-m$. Then $f_m\in\F_{(1,1\mid1)}$,
and $f_m$ converges to the identity uniformly on compact subsets
of $\R$, since $f_m(x)=x$ on $[-m,\infty)$. However,
$|f_m(x)-x|=|x|-m$ for $x<-m$, so
$$
d_\wsup(f_m,\mathrm{id})
=\sup_{x<-m}\frac{|x|-m}{1+|x|}=1
\qquad\text{for every }m,
$$
and $f_m$ does not converge to the identity in the weighted sup
norm. The compact-open limit is not in $\F_{(1,1\mid1)}$: every
function $c\relu(wx+b)+d$ is constant on some half-line, whereas
the identity is not. Thus $\F_{(1,1\mid1)}$ is not closed in the
compact-open topology on $\R$.
\end{example}

\subsection{Continuity of the realization map}

Recall that $\F_A$ denotes the realization space of an
architecture $A=(n_0,\ldots,n_d\mid n_{d+1})$, that is, the set
of functions $\R^{n_0}\to\R^{n_{d+1}}$ realized by networks of
architecture $A$. Every such function is CPWL, hence Lipschitz,
and so has at most linear growth; thus
$\F_A\subseteq C_\wsup(\R^{n_0},\R^{n_{d+1}})$. Henceforth, unless
otherwise specified, we equip $\F_A$ with the subspace topology
induced by the weighted sup norm.

For compact domains, continuity and local Lipschitz continuity of the
realization map were established in
\cite[Proposition~4.1]{petersen2021topological}. The following is the
analogous statement for the weighted sup norm on all of $\R^{n_0}$.

\begin{proposition} \label{realizationmapcontinuous}
Fix an architecture $A=(n_0,\ldots,n_d\mid n_{d+1})$. The
realization map $\rho_A:\Theta_A\to\F_A$ is continuous with
respect to the weighted sup norm. In particular, on any bounded
subset of $\Theta_A$, $\rho_A$ is Lipschitz.
\end{proposition}
\begin{proof}

Fix any bounded subset of parameter space $S \subset \Theta_A$ and consider $\theta, \theta' \in S$.  
Write 
$$\theta = (W^{(\ell)}, b^{(\ell)})_{\ell=1}^d, \quad 
\theta' = (W'^{(\ell)}, b'^{(\ell)})_{\ell=1}^d.$$
Then boundedness implies there exists some $M > 0$ such that all parameters in $S$ satisfy $\|W^{(\ell)}\|_{op} \leq M$ and $\|b^{(\ell)}\| \leq M$ for all $1 \leq \ell \leq d$.  (Here, $|| \cdot ||_{op}$ denotes operator norm.)

Define the forward iterates
$x^{(0)} = x'^{(0)} = x,$
$$\quad x^{(\ell)}_\theta = \sigma\big(W^{(\ell)} x^{(\ell-1)}_\theta + b^{(\ell)}\big), \quad 
\quad x'^{(\ell)}_\theta = \sigma\big(W'^{(\ell)} x'^{(\ell-1)}_\theta + b'^{(\ell)}\big),
$$
so $f_\theta(x)=x^{(d)}_\theta$ and $f_{\theta'}(x)=x'^{(d)}_\theta.$ 
We will evaluate the differences $$\delta^{(\ell)}(x) := x^{(\ell)}_\theta(x) - x^{(\ell)}_{\theta'}(x).$$

Because $\sigma$ is $1$-Lipschitz (i.e. $||\sigma(x) - \sigma(y)|| \leq ||x-y||$ for all $x, y$), 

\begin{align*}
\| \delta^{(\ell)}(x)\| &  = \| \sigma\big(W^{(\ell)} x^{(\ell-1)}_\theta + b^{(\ell)}\big) - \sigma\big(W'^{(\ell)} x^{(\ell-1)}_{\theta'} + b'^{(\ell)}\big) \|  \\
& \leq 
\|W^{(\ell)} x^{(\ell-1)}_\theta + b^{(\ell)} - W'^{(\ell)} x^{(\ell-1)}_{\theta'} - b'^{(\ell)} \|  \\
& = \| W^{(\ell)}(x^{(\ell-1)}_\theta - x^{(\ell-1)}_{\theta'}) + (W^{(\ell)} - W'^{(\ell)}) x^{(\ell-1)}_{\theta'} + b^{(\ell)} - b'^{(\ell)} \| \\
& \leq  
\|W^{(\ell)}\|_{op} \cdot \| \delta^{(\ell-1)}(x)\| + \|W^{(\ell)}-W'^{(\ell)}\|_{op}\cdot \|x^{(\ell-1)}_{\theta'}\| + \|b^{(\ell)} - b'^{(\ell)}\|
\end{align*}

We claim there exists $C_0 \in \mathbb{R}$ such that for all $\ell$ and all $x$,
\[
\|x^{(\ell)}_\theta(x)\| \le C_0 (1+\|x\|),
\]
uniformly over $\|\theta\|\le M$. Indeed, using $\|\sigma(z)\|\le \|z\|$ and the operator norm,
\[
\|x^{(\ell)}_\theta\|
\le \|W^{(\ell)}\|_{\mathrm{op}} \|x^{(\ell-1)}_\theta\| + \|b^{(\ell)}\|
\le M \|x^{(\ell-1)}_\theta\| + M,
\]
and iterating this affine recursion yields the claim.
Thus the bound above becomes 
$$
|| \delta^{(\ell)}(x)||  \leq M  \| \delta^{(\ell-1)}(x)\| + \|W^{(\ell)}-W'^{(\ell)}\| C_0(1+ \|x^{(\ell-1)}_{\theta'}\|) + \|b^{(\ell)} - b'^{(\ell)}\|.$$
Dividing by $1+\|x\|$ yields
\[
\frac{\|\delta^{(\ell)}(x)\|}{1+\|x\|}
\le M \frac{\|\delta^{(\ell-1)}(x)\|}{1+\|x\|}
+ C_0 \|W^{(\ell)}-W'^{(\ell)}\|
+ \frac{\|b^{(\ell)} - b'^{(\ell)}\|}{1+\|x\|}.
\]
Taking the supremum over $x$ yields
\[
\|\delta^{(\ell)}\|_{\mathrm{wsup}}
\le M \|\delta^{(\ell-1)}\|_{\mathrm{wsup}}
+ C_0 \|W^{(\ell)}-W'^{(\ell)}\|
+ \|b^{(\ell)} - b'^{(\ell)}\|.
\]
Iterating this inequality over $\ell=1,\dots,d$ gives
\[\| \delta^{(d)}\|_\textrm{wsup} =
\|f_\theta - f_{\theta'}\|_{\mathrm{wsup}}
\le C \|\theta - \theta'\|
\]
for some constant $C$.
\end{proof}

\subsection{Interpretation of the wsup norm in terms of a compactified model}

In this subsection we explain that the weighted sup norm is, up to a change of coordinates, simply the sup norm on a suitably compactified model: conjugating $\mathbb{R}^n$ onto the open unit ball by a homeomorphism turns the weight $(1+\|x\|)^{-1}$ into the constant weight $1$.

We will denote the open and closed unit balls in $\mathbb{R}^n$ as
\begin{align*}
B^n & \coloneqq \{x \in \mathbb{R}^n : \|x\| < 1\}, \\ 
\overline{B^n} & \coloneqq \{x \in \mathbb{R}^n : \|x\| \leq 1\}
\end{align*}
Define $h : \mathbb{R}^n \to B^n$ by
\[
h(x) = \frac{x}{1+\|x\|}.
\]
Note that $h$ is a homeomorphism with inverse $h^{-1} : B^n \to \mathbb{R}^n$ given by
\[
h^{-1}(y) = \frac{y}{1-\|y\|}.
\]
 
\begin{lemma} \label{l:compactifiedmodeljustificaiton}
Let $f:\mathbb R^n\to\mathbb R^m$ be a continuous function admitting a
finite cover by polyhedra on each of which $f$ is affine.  Then the map $\breve f: B^n \to \mathbb{R}^m$ defined by 
$$
\breve f(y) \coloneqq \frac{f(h^{-1}(y))}{1+\|h^{-1}(y)\|} = (1-\|y\|)\, f\!\left(\frac{y}{1-\|y\|}\right)$$
admits a unique continuous extension to the closed ball $\overline{B^n}$.
\end{lemma}

\begin{proof}
Denote the convex polyhedra by $P_1,\ldots, P_r$ and the affine functions on these polyhedra by $f|_{P_i}(x) = A_ix + b$. 
First, note that $f$ is $M$-Lipschitz on $\mathbb{R}^n$ for $M \coloneqq \max_{i} ||A_i||_{op}$, the maximum of the operator norms of the linear pieces of the affine functions. 

Note that for any vector $u \in S^{n-1}$, the limit $\lim_{t \to \infty} \frac{f(tu)}{1+tu}$ is a well-defined point in $\mathbb{R}^m$.  Indeed, there is some index $i_u \in \{1,\ldots,r\}$ such that $tu \in P_{i_u}$ for all arbitrarily large $t >0$; then 
$$\lim_{t \to \infty} \frac{f(tu)}{1+t} = \lim_{t \to \infty} \frac{f(tu)}{t}  =  A_{i_u} u \in \mathbb{R}^m.$$ 
Define $\widetilde{f}:\overline{B^n} \to \mathbb{R}^m$ by 
$$ \widetilde{f}(x) = \begin{cases}
\breve f(x) & \textrm{if } ||x|| < 1,\\
A_{i_u} u & \textrm{if } ||x||=1.
\end{cases}$$

On the open ball $B^n$, $\widetilde{f} = \breve f$ is expressed as a composition of continuous functions; hence the restriction of $\widetilde{f}$ to the open ball is continuous.

For any $u,v \in S^n$, we have 
\begin{align} \| \widetilde{f}(u) - \widetilde{f}(v) \| 
& = \left \|
\lim_{t\to\infty} \frac{f(tu)}{1+t} - \lim_{t \to \infty} \frac{f(t)}{1+t} \right\| \\ 
& =  \lim_{t \to \infty} \frac{1}{1+t} \left\|
f(tu)-f(tv)\right\| \\
&\leq
\lim_{t\to\infty}
\frac{M\|tu-tv\|}{1+t} \\
&= 
M\|u-v\|. \label{eq:MLipschitz}
\end{align}
The same estimate holds on every sphere of radius $s<1$:
with $t=s/(1-s)$ we have $1+t=1/(1-s)$ and $h^{-1}(sw)=tw$,
so for $u,v\in S^{n-1}$,
$$
\|\widetilde f(sv)-\widetilde f(su)\|
=\frac{\|f(tv)-f(tu)\|}{1+t}
\le\frac{Mt}{1+t}\|u-v\|
\le M\|u-v\|.
$$

Also, for $s \in [0,1)$ and $u \in S^{n-1}$ we have 

\[
h^{-1}(su) = \frac{su}{1-\|su\|} = \frac{s}{1-s}u
\]

and hence 

\begin{align*}
\|\widetilde{f}(su) - \widetilde{f}(u)\|
&= \| \breve{f}(su) - \widetilde{f}(u) \|  \\
&=  \left\| \frac{f(h^{-1}(su))}{1+\|h^{-1}(su)\|}
- \lim_{t \to \infty} \frac{f(tu)}{1+t} \right\|  \\
&= \left\| \frac{f\left(\frac{s}{1-s} u \right)}{1+\frac{s}{1-s}}
- \lim_{t \to \infty} \frac{f(tu)}{1+t} \right\|.
\end{align*}  It follows that
\begin{equation} \label{eq:raysconverge}
\lim_{s \nearrow 1} \widetilde{f}(su) = \widetilde{f}(u)
\end{equation} for all $u \in S^{n-1}$. 

Consider any point $u \in \partial S^{n-1}$ and any sequence $(x_i)_{i \in \mathbb{N}}$ in $\overline{B^n}$ satisfying $x_i \to u$.   Then we can write $x_i = s_i v_i$ for some sequence of real numbers in $[0,1]$ satisfying $s_i \to 1$ and unit vectors $v_i \in S^{n-1}$ satisfying $v_i \to u$. 

\begin{align*} \| 
\widetilde{f}(x_i) - \widetilde{f}(u) \|  \leq \| \widetilde{f}(s_iv_i) - \widetilde{f}(s_i u) \|  + \|\widetilde{f}(s_iu) - \widetilde{f}(u)\|
\end{align*}

Equations \eqref{eq:MLipschitz} and  \eqref{eq:raysconverge} above show for sufficiently large $i$,  both terms on the right side of the above expression can be made arbitrarily small. 
Thus $\widetilde{f}$ is continuous on $\overline{B}^n$.

Finally, the continuous extension is unique because $B^n$ is dense in
$\overline{B^n}$.
\end{proof}
\color{black}

Lemma \ref{l:compactifiedmodeljustificaiton} justifies the following definition. 
\begin{definition}
Let $f:\mathbb R^n\to\mathbb R^m$ be a continuous function admitting a
finite cover by polyhedra on each of which $f$ is affine. We define the \emph{compactified model of $f$} to be the map $$\widetilde{f}:\overline{B^n} \to \mathbb{R}^m$$
that is the unique continuous extension to $\overline{B^n}$ of the function 
$\breve f: B^n \to \mathbb{R}^m$, 
$$
\breve f(y) \coloneqq \frac{f(h^{-1}(y))}{1+\|h^{-1}(y)\|}. $$
\end{definition}

\begin{lemma}
 \label{lem:wsup-compactified}
Let $f:\mathbb R^n\to\mathbb R^m$ be a continuous function admitting a
finite cover by polyhedra on each of which $f$ is affine. Then the compactified $\widetilde f:\overline{B^n} \to \mathbb{R}^m$ satisfies 
\[
\|f\|_{\mathrm{wsup}, \mathbb{R}^n} = \|\widetilde f\|_{\mathrm{sup}, \overline{B^n}}.
\]
\end{lemma}

\begin{proof}
Since $h:\mathbb{R}^n \to B^n$ is a homeomorphism, as a point $y$ ranges over the open ball $B^n$, the corresponding point $x = h^{-1}(y)$ ranges bijectively over $\mathbb{R}^n$.  Thus 
\[
\sup_{y \in B^n} \|\widetilde f(y)\| = \sup_{x \in \mathbb{R}^n} \frac{\|f(x)\|}{1+\|x\|} = \|f\|_{\mathrm{wsup}, \mathbb{R}^n}.
\]
Every point in the closed ball $\overline{B^n}$ is an accumulation point of the open ball $B^n$, and $\widetilde{f}$ is continuous, so 

\[
\sup_{y \in B^n} \|\widetilde f(y)\| = \sup_{y \in \overline{B^n}} \|\widetilde f(y)\|.\]
The expression at right above is the definition of $\|\widetilde f\|_{\mathrm{sup}, \overline{B^n}}$. 
\end{proof}

Lemma  \ref{lem:wsup-compactified} has the following immediate corollary. 
\begin{corollary}
A sequence of CPWL functions $(f_n)_{n \in \mathbb{N}}$ in $C_{\textrm{wsup}}(\mathbb{R}^n,\mathbb{R}^m)$ converges with respect to the weighted sup norm metric if and only if the sequence of compactified models $(\widetilde{f}_n)_{n \in \mathbb{N}}$ converges with respect to the sup norm metric.  
\end{corollary}

\begin{remark}
Not every function in $C_{\mathrm{wsup}}$ admits a continuous
compactified model on $\overline{B^n}$. For example,
$f(x)=x\sin x$ belongs to $C_{\mathrm{wsup}}(\R,\R)$,
but for $0<y<1$,
\[
\breve f(y)=y\sin\left(\frac{y}{1-y}\right),
\]
which has no limit as $y\to1$. Thus $\breve f$
does not extend continuously to the closed unit ball.
\end{remark}

\section{Shallow networks}
\label{s:shallow}

\subsection{Functions realizable by shallow networks}
\label{ss:shallow-characterization}

We characterize the functions $G:\R^n\to\R^p$ realized by shallow
networks of architecture $(n,N\mid p)$, that is, functions of the form
\begin{equation}
\label{eq:shallow-form}
G(x)=B+\sum_{i=1}^N c_i\relu(w_i^Tx+b_i),
\qquad c_i,B\in\R^p,\quad w_i\in\R^n,\quad b_i\in\R.
\end{equation}

The exact representability of CPWL functions by shallow
ReLU networks has been studied from several perspectives.
McCarty~\cite{mccarty2023piecewise} proved that a CPWL
function admitting a finite-cost infinite-width shallow
ReLU representation also admits a finite-width
representation. Fu~\cite{fu2025toric} developed a
toric-geometric approach to ReLU realizability, obtaining
a characterization for bias-free shallow scalar networks
with rational parameters in terms of intersection numbers.
For scalar-valued shallow networks, Dereich and
Kassing~\cite{dereich2022minimal} characterized the
minimum number of hidden neurons needed to realize a
given function. We extend this characterization to
vector-valued outputs.

The nonlinear part of a shallow network is determined by its visible
hyperplanes and the Jacobian jumps across them. After isolating this
part, the problem is to determine how economically the remaining affine
map can be realized.

\subsubsection{Visible bends and the bend decomposition}

An affine function $L(x)=u^Tx+\beta$, with $u\ne0$, determines a
hyperplane $H=\{L=0\}$, cooriented from $\{L<0\}$ to $\{L>0\}$.

\begin{definition}[visible bend presentation]
\label{def:visible-bends}
A \emph{visible bend presentation} of a continuous function
$G:\R^n\to\R^p$ is a list $\mathcal L=(L_1,\ldots,L_q)$ of affine
functions $L_j(x)=u_j^Tx+\beta_j$, with $u_j\ne0$ and pairwise
distinct zero hyperplanes $H_j=\{L_j=0\}$, such that
\begin{enumerate}
\item[\textup{(i)}] $G$ is affine on each chamber of the arrangement;
\item[\textup{(ii)}] for each $j$, the Jacobian jump of $G$ across
every codimension-one cell of the arrangement contained in $H_j$,
cooriented by $L_j$, equals the same nonzero matrix
$J_j\in\R^{p\times n}$.
\end{enumerate}
The hyperplanes $H_j$ are the \emph{visible bends} and the matrices
$J_j$ are the \emph{visible jumps}. We allow $q=0$, in which case
$G$ is affine.
\end{definition}

For any visible bend presentation, the set of points where $G$ is
not locally affine is exactly $\bigcup_{j=1}^qH_j$. Indeed, $G$
is locally affine off this union. Conversely, the nonzero jump across
$H_j$ makes $G$ non-affine at each point of $H_j$ outside the other
hyperplanes. Such points are dense in $H_j$, and the set where $G$
is not locally affine is closed. Therefore the visible hyperplanes
are uniquely determined by $G$.

By Lemma~\ref{l:jump-normal}, applied to each coordinate, the jump
across $H_j$ has the form
\[
J_j=v_ju_j^T
\]
for a unique nonzero vector $v_j\in\R^p$.

\begin{lemma}[bend decomposition]
\label{l:bend-decomposition}
If $G$ admits a visible bend presentation as above, there are unique
$A\in\R^{p\times n}$ and $B\in\R^p$ such that
\begin{equation}
\label{eq:bend-decomposition}
G(x)=\frac12\sum_{j=1}^q v_j|L_j(x)|+Ax+B.
\end{equation}
\end{lemma}

\begin{proof}
The function $F=\frac12\sum_jv_j|L_j|$ has Jacobian jump
$v_ju_j^T$ across each codimension-one cell contained in $H_j$.
Both $F$ and $G$ are affine on the top-dimensional cells of the
arrangement, so Lemma~\ref{l:determined-by-jumps} implies that
$G-F$ is affine. Uniqueness follows immediately.
\end{proof}

We call the first sum in \eqref{eq:bend-decomposition} the
\emph{bend part} of $G$ and $Ax+B$ its \emph{affine residual}.
These are intrinsic to $G$: replacing $L_j$ by $\lambda_jL_j$,
$\lambda_j\ne0$, replaces $v_j$ by $v_j/|\lambda_j|$ and leaves
$v_j|L_j|$ unchanged. 

\subsubsection{Hyperplane blocks}

Consider a realization \eqref{eq:shallow-form}. Discard neurons with
zero output weight and absorb constant neurons into $B$.
Group the remaining neurons according to their zero hyperplanes.
For each hyperplane $K$, the sum of the neurons with zero hyperplane
$K$ is its \emph{hyperplane block}, denoted $\Phi_K$.

Fix an affine defining function $L_K(x)=u_K^Tx+\beta_K$ for $K$.
Every neuron in the block has an affine input equal to a nonzero
scalar multiple of $L_K$. By positive homogeneity of ReLU, the block
can therefore be written
\[
\Phi_K=a_K^+\relu(L_K)+a_K^-\relu(-L_K)
\]
for some $a_K^+,a_K^-\in\R^p$. Equivalently,
\begin{equation}
\label{eq:block-decomposition}
\Phi_K=\frac12v_K|L_K|+w_KL_K,
\qquad
v_K=a_K^++a_K^-,\quad
w_K=\frac12(a_K^+-a_K^-).
\end{equation}
We call $v_K$ and $w_K$ the \emph{bend coefficient} and
\emph{linear coefficient} of the block, relative to $L_K$.

The block has Jacobian jump $v_Ku_K^T$ across $K$.
At a point of $K$ lying on no other hyperplane of the network,
all other blocks are locally affine. It follows that $K$ is a visible
bend of $G$ exactly when $v_K\ne0$. In particular, the non-affine
locus of $G$ is the union of these hyperplanes, so the set of visible
hyperplanes is determined by $G$.

The number of neurons in a block constrains its linear coefficient.
If $v_K\ne0$ and the block has just one neuron, then
\begin{equation}
\label{eq:single-block}
\Phi_K=v_K\relu(\epsilon L_K)
=\frac12v_K|L_K|+\frac{\epsilon}{2}v_KL_K,
\qquad \epsilon\in\{-1,1\}.
\end{equation}
Thus $w_K=\epsilon v_K/2$. By contrast, with two neurons on $K$
we can prescribe $w_K$ arbitrarily, since
\begin{equation}
\label{eq:two-neuron-block}
\frac12v_K|L_K|+w_KL_K
=\left(\frac12v_K+w_K\right)\relu(L_K)
 +\left(\frac12v_K-w_K\right)\relu(-L_K).
\end{equation}
Finally, if $v_K=0$, the block is affine: $\Phi_K=w_KL_K$.
A nonzero such block needs at least two neurons, and two suffice by
\begin{equation}
\label{eq:cancelling-pair}
w_KL_K=w_K\relu(L_K)-w_K\relu(-L_K).
\end{equation}

Thus every visible hyperplane costs at least one neuron. An additional
neuron on that hyperplane makes its linear coefficient arbitrary;
a nonzero invisible block costs at least two neurons.

\subsubsection{Residual cost and characterization}

Fix the bend decomposition \eqref{eq:bend-decomposition}.
Let $S\subseteq\{1,\ldots,q\}$ be the visible hyperplanes on which
we allow two neurons. On each remaining visible hyperplane, the
single neuron must have one of the orientations
$\epsilon_j\in\{-1,1\}$, giving linear coefficient
$\epsilon_jv_j/2$ by \eqref{eq:single-block}.
Put
\[
U_S=\operatorname{span}\{u_j:j\in S\},\qquad
M_{S,\epsilon}
=A-\frac12\sum_{j\notin S}\epsilon_jv_ju_j^T,
\]
where $U_\varnothing=\{0\}$. The two-neuron blocks on the
hyperplanes indexed by $S$ can realize any linear map whose rows
lie in $U_S$. The remaining linear map has rank
$\operatorname{rank}(M_{S,\epsilon}|_{U_S^\perp})$ when restricted
to $U_S^\perp$, and can be represented by that many cancelling
pairs. This leads to the following cost.

\begin{definition}[residual cost]
\label{def:residual-cost}
For a function $G$ admitting a visible bend presentation, define the \emph{residual cost} to be
\begin{equation}
\label{eq:residual-cost}
\tau(G)=
\min_{\substack{S\subseteq\{1,\ldots,q\}\\
                  \epsilon_j\in\{-1,1\},\ j\notin S}}
\left\{|S|+2\operatorname{rank}
\bigl(M_{S,\epsilon}|_{U_S^\perp}\bigr)\right\}.
\end{equation}
\end{definition}

This is independent of the chosen visible bend presentation:
reordering the hyperplanes changes nothing, and rescaling or
reversing a defining function preserves $A$, each $U_S$, and the
set of possible signed matrices $\epsilon_jv_ju_j^T$.

\begin{theorem}[characterization of shallow realizable functions; Theorem~\ref{mt:characterization}]
\label{t:shallow-characterization}
Let $G:\R^n\to\R^p$ be continuous and $N\in\mathbb N$.
Then $G\in\F_{(n,N\mid p)}$ if and only if $G$ admits a visible bend
presentation with $q$ visible bends and
\[
q+\tau(G)\le N.
\]
\end{theorem}

\begin{proof}
Suppose first that $G$ admits a visible bend presentation and
$q+\tau(G)\le N$. Choose $S$ and $\epsilon$ attaining the minimum
in \eqref{eq:residual-cost}, and write $U=U_S$,
$M=M_{S,\epsilon}$, and $r=\operatorname{rank}(M|_{U^\perp})$.
If $P_U$ denotes orthogonal projection onto $U$, then
\[
MP_U=\sum_{j\in S}w_ju_j^T
\]
for suitable $w_j\in\R^p$. Similarly, $MP_{U^\perp}$ has rank $r$,
so we may write
\[
MP_{U^\perp}=\sum_{k=1}^r a_kb_k^T,
\qquad a_k\in\R^p,\quad b_k\in\R^n.
\]

For $j\notin S$, take the single-neuron block
$v_j\relu(\epsilon_jL_j)$. For $j\in S$, take the block
$\frac12v_j|L_j|+w_jL_j$, realized by two neurons using
\eqref{eq:two-neuron-block}. Finally, realize each linear term
$a_k(b_k^Tx)$ by the cancelling pair
$a_k\relu(b_k^Tx)-a_k\relu(-b_k^Tx)$.
The sum of these blocks has precisely the bend part and linear
part of $G$; its constant term can be adjusted by the output bias.
The total number of neurons is at most
\[
(q-|S|)+2|S|+2r=q+\tau(G)\le N.
\]
Adding zero-output neurons if necessary gives a realization with
exactly $N$ neurons.

Conversely, suppose $G$ has a realization with $N$ neurons.
After discarding zero-output and constant neurons, let $m\le N$
be the number remaining, and group them into hyperplane blocks.
The blocks with nonzero bend coefficient give the visible bend
presentation $H_1,\ldots,H_q$ of $G$.

Let $m_j$ be the number of neurons in the block on $H_j$, and put
$S=\{j:m_j\ge2\}$. For $j\notin S$, the block is
$v_j\relu(\epsilon_jL_j)$ for some sign $\epsilon_j$.
For $j\in S$, it has the form
$\frac12v_j|L_j|+w_jL_j$. Every invisible block $K$ is affine,
with the form $w_KL_K$; let $t$ be the number for which $w_K\ne0$.
Comparing the linear parts of the blocks with
\eqref{eq:bend-decomposition} gives
\begin{equation}
\label{eq:residual-rank-bound}
M_{S,\epsilon}
=\sum_{j\in S}w_ju_j^T
 +\sum_{\substack{K\text{ invisible}\\w_K\ne0}}w_Ku_K^T.
\end{equation}
The first sum vanishes on $U_S^\perp$, and the second has at most
$t$ rank-one terms. Hence
\[
\operatorname{rank}(M_{S,\epsilon}|_{U_S^\perp})\le t,
\qquad\text{so}\qquad \tau(G)\le |S|+2t.
\]
Each visible block uses at least one neuron, with at least one
additional neuron for every $j\in S$. Each of the $t$ nonzero
invisible blocks requires at least two neurons. Therefore
\[
N\ge m\ge q+|S|+2t\ge q+\tau(G),
\]
as claimed.
\end{proof}

\subsubsection{Scalar output}

When $p=1$, every restricted linear map in
\eqref{eq:residual-cost} has rank at most one, so
$\tau(G)\in\{0,1,2\}$. A set $S$ with $|S|\ge2$ cannot improve
upon $S=\varnothing$, for which the cost is at most two.
For a complete choice of signs $\epsilon\in\{-1,1\}^q$, set
\[
A_\epsilon=A-\frac12\sum_{j=1}^q\epsilon_jv_ju_j^T.
\]
The characterization reduces to the following trichotomy.

\begin{corollary}[{\cite{dereich2022minimal}}]
\label{cor:scalar-characterization}
Let $G:\R^n\to\R$ admit a visible bend presentation with
$q$ visible bends. The minimum number of hidden neurons realizing
$G$ is
\begin{enumerate}
\item $q$, if $A_\epsilon=0$ for some $\epsilon$;
\item $q+1$, if the preceding condition fails but some
$A_\epsilon$ is a nonzero multiple of $u_j^T$ for some $j$;
\item $q+2$, otherwise.
\end{enumerate}
\end{corollary}

\begin{proof}
The cost is zero exactly when $S=\varnothing$ and
$A_\epsilon=0$. If the cost is not zero, it is one exactly when
$S=\{j\}$ and $M_{S,\epsilon}$ vanishes on $U_S^\perp$ for
some $j$ and some signs outside $S$. Extending these signs to
$\epsilon_j$ arbitrarily, this condition is equivalent to
$A_\epsilon\in\operatorname{span}\{u_j^T\}$. Under the failure of
the first case, $A_\epsilon$ is necessarily nonzero. In all other
cases the cost is two.
\end{proof}

\begin{example}
\label{ex:residual-cost}
For vector-valued output, the choice of $S$ may matter. Let
$G:\R^2\to\R^2$ be given by
\[
G(x)=\bigl(x_1+\relu(x_1),\ x_2+\relu(x_2)\bigr).
\]
Its two visible bends are $\{x_1=0\}$ and $\{x_2=0\}$,
with $L_j(x)=x_j$ and $v_j=u_j=e_j$. The bend decomposition is
\[
G(x)=\frac12e_1|x_1|+\frac12e_2|x_2|+\frac32x.
\]
For $S=\varnothing$, the matrix $M_{S,\epsilon}$ has rank two
for every $\epsilon$, giving cost four. For $|S|=1$, its
restriction to $U_S^\perp$ has rank one, giving cost three.
For $S=\{1,2\}$, the residual restriction vanishes, giving cost
two. Thus $\tau(G)=2$, and the minimum number of neurons is four.
One such realization is
\[
G(x)=\bigl(2\relu(x_1)-\relu(-x_1),\
           2\relu(x_2)-\relu(-x_2)\bigr).
\]
\end{example}

\subsection{Closedness for shallow scalar output}
\label{ss:shallow-closedness}

We now apply the preceding results on limits of CPWL functions
to prove that shallow scalar-valued realization spaces are closed
in the weighted sup norm. Corollary~\ref{cor:network-limits}
ensures that such limits are CPWL, while
Corollary~\ref{cor:limit-jumps} describes the behavior of
gradient jumps when affine pieces merge. In the shallow setting,
these jumps arise from neuron hyperplanes. By grouping neurons
whose hyperplanes converge to the same limiting hyperplane,
we will show that their jumps combine to give the visible bends
of the limiting function. Some jumps may cancel, and some
hyperplanes may escape to infinity.

Once the visible bends have been identified, the
characterization in Theorem~\ref{t:shallow-characterization}
reduces closedness to a counting argument. If the limit $f$
has $q$ visible bends, we must show that
$q+\tau(f)\le N$. Since the residual cost for scalar output
is at most two, the only delicate cases are those in which
fewer than two neurons remain after assigning one to each
visible bend.

\subsubsection{Limiting neuron clusters}

\begin{lemma}[limiting neuron clusters and visible bends]
\label{l:limiting-clusters}
Suppose that
\[
f_m(x)=b_m+\sum_{i=1}^N c_{m,i}\relu(L_{m,i}(x)),
\qquad
L_{m,i}(x)=u_{m,i}^Tx+\beta_{m,i},
\]
converges to $f:\R^n\to\R$ in the weighted sup norm.  Assume that
\[
\|u_{m,i}\|^2+\beta_{m,i}^2=1,
\qquad
(u_{m,i},\beta_{m,i})\longrightarrow(u_i,\beta_i)
\]
for every $i\in\{1,\ldots,N\}$.

Let $I_{\mathrm{esc}}=\{i:u_i=0\}$ and partition its complement
into equivalence classes $I_1,\ldots,I_r$ under
\[
i\sim k
\quad\Longleftrightarrow\quad
(u_i,\beta_i)=\pm(u_k,\beta_k).
\]
For each $j$, choose a representative $(u_j,\beta_j)$ of $I_j$ and
set
\[
L_j(x)=u_j^Tx+\beta_j,\qquad H_j=\{L_j=0\}.
\]
For $i\in I_j$, define $\epsilon_i\in\{-1,1\}$ by
$(u_i,\beta_i)=\epsilon_i(u_j,\beta_j)$.

Then $f$ is affine on each chamber of the arrangement
$\{H_1,\ldots,H_r\}$.  Moreover, for every $j$ there is a unique
$a_j\in\R$ such that the gradient jump of $f$ across every
codimension-one cell of the arrangement contained in $H_j$,
cooriented by $L_j$, is
\begin{equation}\label{eq:cluster-jump}
 a_ju_j
 =\lim_{m\to\infty}\sum_{i\in I_j}
       \epsilon_i c_{m,i}u_{m,i}.
\end{equation}
Consequently, the functions $L_j$ with $a_j\ne0$ form a visible
bend presentation of $f$.
\end{lemma}

\begin{proof}
Outside the union of the limiting hyperplanes
$H_1,\ldots,H_r$, every point has a neighborhood
on which, for all sufficiently large $m$, each
$L_{m,i}$ has constant sign. Thus $f_m$ is affine
on that neighborhood. By Lemma~\ref{l:affine-limits},
$f$ is locally affine there, and hence affine
on each chamber of the limiting arrangement.

Fix a codimension-one cell $P\subset H_j$.
Choose small open balls $B^-$ and $B^+$ in the
two chambers adjacent across $P$, sufficiently
close to $P$ that the segment joining their
centers meets no other limiting hyperplane.
For all sufficiently large $m$, the functions
$f_m$ are affine on both balls.

Every neuron outside $I_j$ has the same activation
state on $B^-$ and $B^+$, while each neuron in
$I_j$ changes its activation state according
to its limiting orientation. Consequently,
\[
\nabla f_m|_{B^+}-\nabla f_m|_{B^-}
=
\sum_{i\in I_j}\epsilon_i c_{m,i}u_{m,i}.
\]
By Lemma~\ref{l:affine-limits}, the left-hand
side converges to $\delta_f(P)$.
This is the same telescoping mechanism as in
Corollary~\ref{cor:limit-jumps}.

The expression on the right depends only on
$I_j$, not on $P$. Moreover, continuity of $f$
implies that $\delta_f(P)$ is normal to $H_j$.
Thus there is a unique $a_j\in\R$ such that
\[
\delta_f(P)=a_ju_j
=
\lim_{m\to\infty}
\sum_{i\in I_j}\epsilon_i c_{m,i}u_{m,i}.
\]

Finally, set $\mathcal A=\{j:a_j\ne0\}$.
The functions
\[
f
\qquad\text{and}\qquad
\sum_{j\in\mathcal A}a_j\relu(L_j)
\]
have the same gradient jumps across every
codimension-one cell of the limiting arrangement.
By Lemma~\ref{l:determined-by-jumps}, their
difference is affine. Therefore
$\{L_j:j\in\mathcal A\}$ is a visible bend
presentation of $f$.
\end{proof}

The next observation explains why the topology matters.  A single
neuron, regardless of the size of its parameters, cannot approximate
a nonconstant affine function in the weighted sup norm.

\begin{lemma}
\label{l:one-neuron}
For every $b,c,\beta,d\in\R$ and $u,a\in\R^n$,
\[
\bigl\|b+c\relu(u^Tx+\beta)-(a^Tx+d)\bigr\|_\wsup
\ge\|a\|.
\]
Consequently, a weighted sup norm limit of functions realized by
one hidden neuron and an output bias, if affine, must be constant.
\end{lemma}

\begin{proof}
There is nothing to show if $a=0$.  Otherwise choose
$v\in\{a/\|a\|,-a/\|a\|\}$ so that $u^Tv\le0$.
The function $t\mapsto\relu(tu^Tv+\beta)$ is bounded for $t\ge0$.
Consequently,
\[
\begin{aligned}
\bigl\|b+c\relu(u^Tx+\beta)-(a^Tx+d)\bigr\|_\wsup
&\ge\lim_{t\to\infty}
\frac{|b+c\relu(tu^Tv+\beta)-t\,a^Tv-d|}{1+t}\\
&=|a^Tv|=\|a\|.
\end{aligned}
\]
\end{proof}

\subsubsection{Proof of closedness}

\begin{theorem}[closedness of scalar shallow realization spaces; Theorem~\ref{mt:shallowclosed}]
\label{t:shallow-closed}
For every $n,N\in\mathbb N$, the realization space
$\F_{(n,N\mid1)}$ is closed in $C_\wsup(\R^n,\R)$.
\end{theorem}

\begin{proof}
Let $f_m\in\F_{(n,N\mid1)}$ converge to $f$ in the
weighted sup norm. Choose representations
\[
f_m(x)=b_m+\sum_{i=1}^N c_{m,i}\relu(L_{m,i}(x)),
\qquad
L_{m,i}(x)=u_{m,i}^Tx+\beta_{m,i}.
\]
Absorbing constant neurons into the output bias and
padding with zero-output neurons if necessary, we may
assume that all $L_{m,i}$ are nonconstant and that there
are exactly $N$ terms. By positive homogeneity, we may
normalize
\[
\|u_{m,i}\|^2+\beta_{m,i}^2=1.
\]
After passing to a subsequence, assume that
$(u_{m,i},\beta_{m,i})\to(u_i,\beta_i)$ for every $i$.

Apply Lemma~\ref{l:limiting-clusters}, and let
$I_{\mathrm{esc}}$ and $I_1,\ldots,I_r$ be the
escaping indices and finite limiting clusters
defined there. For each $j$, let $H_j$ be the
limiting hyperplane and $a_j$ its jump coefficient.
Set
\[
\mathcal A:=\{j\in\{1,\ldots,r\}:a_j\ne0\},
\qquad q:=|\mathcal A|.
\]
By Lemma~\ref{l:limiting-clusters}, the hyperplanes
$H_j$, $j\in\mathcal A$, form a visible bend
presentation of $f$. By
Theorem~\ref{t:shallow-characterization},
it suffices to show that
$
q+\tau(f)\le N.$

Each visible bend of $f$ corresponds to a nonempty
limiting neuron cluster. Designating one neuron
from each of the $q$ visible clusters leaves
$s:=N-q$ additional neurons, distributed as follows:
\begin{equation}\label{eq:remaining-neurons}
s
=
|I_{\mathrm{esc}}|
+\sum_{j\in\mathcal A}(|I_j|-1)
+\sum_{j\notin\mathcal A}|I_j|.
\end{equation}
Since the output is scalar, $\tau(f)\le2$. Thus the desired
inequality $
q+\tau(f)\le N$ holds whenever $s\ge2$.

It remains to consider $s\le1$. In this situation,
there are only two possibilities: either every visible
cluster consists of a single neuron, or exactly one
visible cluster consists of two neurons, with all
other visible clusters being singletons.
In the first case, we will show that the affine
residual is constant. In the second, we will show
that any nonconstant affine residual has gradient
normal to the hyperplane of the two-neuron cluster,
and can therefore be represented using the one
additional neuron.

We first observe that the neuron associated with a
singleton visible cluster converges in the weighted
sup norm. Indeed, suppose $I_j=\{i\}$ with
$j\in\mathcal A$. Orient $L_j$ so that
\[
(u_{m,i},\beta_{m,i})\longrightarrow(u_j,\beta_j).
\]
The jump formula \eqref{eq:cluster-jump} gives
\[
c_{m,i}u_{m,i}\longrightarrow a_ju_j.
\]
Since $u_j\ne0$, we have $c_{m,i}\to a_j$.
Moreover, the ReLU function is Lipschitz, so
\[
\|\relu(L_{m,i})-\relu(L_j)\|_{\wsup}
\le
\|u_{m,i}-u_j\|+|\beta_{m,i}-\beta_j|.
\]
Consequently,
\begin{equation}\label{eq:singleton-convergence}
c_{m,i}\relu(L_{m,i})
\longrightarrow a_j\relu(L_j)
\qquad\text{in }C_\wsup(\R^n,\R).
\end{equation}

\emph{Case 1: Every visible cluster is a singleton.}
Since $s\le1$, equation \eqref{eq:remaining-neurons}
shows that, apart from the indices belonging to
visible clusters, there is at most one neuron.

Orient all visible hyperplanes to agree with their
singleton neurons, and set
\[
R=f-\sum_{j\in\mathcal A}a_j\relu(L_j).
\]
By Lemma~\ref{l:bend-decomposition}, $R$ is affine.

Subtracting the convergent terms in
\eqref{eq:singleton-convergence} from $f_m$,
we see that $R$ is the weighted sup norm limit of
functions consisting of an output bias and at most
one ReLU neuron. By Lemma~\ref{l:one-neuron}, $R$
must be constant.

Thus $f$ is realized by one neuron for each of its
$q$ visible bends, together with an output bias.
In particular, $\tau(f)=0\le s$.

\emph{Case 2: Some visible cluster contains more than
one neuron.}
Since $s\le1$, equation \eqref{eq:remaining-neurons}
implies that $s=1$ and there is exactly one visible
cluster $I_k$ containing two indices. Every other
visible cluster is a singleton, and there are no
escaping neurons or zero-jump clusters.

Orient the singleton visible hyperplanes as above.
Subtracting their convergent neuron terms gives
\[
h_m:=b_m+\sum_{i\in I_k}c_{m,i}\relu(L_{m,i})
\longrightarrow
h:=f-\sum_{\substack{j\in\mathcal A\\j\ne k}}
a_j\relu(L_j)
\]
in the weighted sup norm.

Choose the orientation of $L_k$ to agree with the
limiting orientation of one of the two neurons in
$I_k$. Since $h$ has exactly one visible bend,
along $H_k$, Lemma~\ref{l:bend-decomposition}
gives
\[
h=a_k\relu(L_k)+R,
\qquad R\text{ affine}.
\]

Let $B$ be a nonempty open ball with
$\overline B\subset\{L_k<0\}$. For sufficiently
large $m$, each of the two neurons in $I_k$ is
either active throughout $B$ or inactive
throughout $B$, according to its limiting
orientation.

If the two neurons have the same limiting
orientation, both are inactive on $B$. Hence
$h_m=b_m$ on $B$, so $h$, and therefore $R$,
is constant on $B$. Since $R$ is affine, it
is constant everywhere.

If the two neurons have opposite limiting
orientations, exactly one neuron is active
on $B$. Denote its index by $i_-$. Then
\[
\nabla h_m|_B=c_{m,i_-}u_{m,i_-}.
\]
Since $h_m\to h=R$ uniformly on $B$,
Lemma~\ref{l:affine-limits} gives
\[
c_{m,i_-}u_{m,i_-}\longrightarrow\nabla R.
\]
But $u_{m,i_-}\to-u_k\ne0$, so
\[
\nabla R\in\operatorname{span}(u_k).
\]

Thus, in either case, the affine residual $R$
has the form
$
R(x)=\lambda L_k(x)+d$
for some $\lambda,d\in\R$.

By \eqref{eq:two-neuron-block}, this residual
can be incorporated using one additional neuron
on $H_k$. Hence $\tau(f)\le1=s$.

\end{proof}


\section{Non-closedness at depth two}
\label{s:nonclosed}

We now show that closedness can fail for a ReLU network with two
hidden layers. Our counterexample has architecture
$A=(2,3,2\mid1)$ and works even on an arbitrary compact domain
with nonempty interior. The construction and its obstruction have
distinct geometric roles. We build a sequence in which the two
second-layer neurons share a large preactivation, causing them to
behave separately in some regions and to cancel in others. The
resulting limit has a nonlinear locus and a collection of gradients
that cannot be realized with the given widths. We describe this
geometry first, then prove the incompatibility with the canonical
polyhedral complex of any putative representing network.

Related non-closedness phenomena have been established by
Mahan, King, and Cloninger~\cite{mahan2021nonclosedness},
who study fixed-architecture neural network realization
spaces in Sobolev topologies for various classes of
activation functions. Their work also investigates the
growth of network parameters along approximating
sequences. Our result concerns ReLU networks in the
uniform topology on compact sets and in the weighted
supremum norm on the whole input space.

\subsection{The sequence and the geometry of its limit}

We will define a sequence
 $\{f_M\}_{M=4}^{\infty}$ of maps  realized by a bias-free network of architecture $A = (2,3,2|1)$.

All the maps $f_M$ will use the same first layer map, consisting of the three neurons $u_1,u_2,u_3:\mathbb{R}^2 \to \mathbb{R}$ given by

\begin{align*}
u_1(x_1,x_2) &=\relu(x_1),\\
u_2(x_1,x_2)&=\relu(x_2),\\
u_3(x_1,x_2)&=\relu(x_1+x_2).
\end{align*}

The second layer map of each $f_M$ will be given by the two neurons 
$v_{M,1},v_{M,2}:\mathbb{R}^3 \to \mathbb{R}$,
\begin{align*}
v_{M,1}(y_1,y_2,y_3) &= \relu\left( M(y_1+y_2-y_3) -4y_1 + y_2 + 2y_3\right),\\
v_{M,2}(y_1,y_2,y_3) &=\relu\left( M(y_1+y_2-y_3) -2y_1 + 4y_2 \right).
\end{align*}
The final affine layer of every $f_M$ will be $(z_1,z_2) \mapsto z_1 - z_2$. Then, for each $M \geq 4$, the map $f_M:\mathbb{R}^2 \to \mathbb{R}$ can be expressed as 

\begin{equation}\label{eq:compact-sequence}
f_M=\relu(Mg_0+h_1)-\relu(Mg_0+h_2)
\end{equation}
where, for convenience, we define $g_0, h_1, h_2: \mathbb{R}^2 \to \mathbb{R}$ by 
\begin{align*}
g_0&=u_1+u_2-u_3,\\
h_1&=-4u_1+u_2+2u_3,\\
h_2&=-2u_1+4u_2.
\end{align*}

By construction, $f_M$ is realized by a bias-free network of architecture $A = (2,3,2|1)$.

The function $g_0:\mathbb{R}^2 \to \mathbb{R}$ is $0$ on the closed first and third quadrants, and strictly positive in the open second and fourth quadrants. 
Consequently, on the closed first and third quadrants, the common 
term $Mg_0$ vanishes  from both second layer neurons $v_{M,1}$ and $v_{M,2}$; in contrast,
on the region where $g_0>0$, both second-layer preactivations eventually become
positive, so $\relu$ acts as the identity, and subtracting their outputs cancels the common term $Mg_0$. Thus
$f_M$ converges pointwise to the function $L:\mathbb{R}^2 \to \mathbb{R}$ defined by
\begin{equation}\label{eq:compact-limit}
L(x)=
\begin{cases}
\relu(h_1(x))-\relu(h_2(x)),&g_0(x)=0,\\
h_1(x)-h_2(x),&g_0(x)>0.
\end{cases}
\end{equation}
In particular, the limiting function uses the two second-layer
ReLUs independently on two full-dimensional regions, but uses
their difference as a linear expression on the intervening regions.

Write $\tau\in[0,2\pi)$ for the polar angle of $x\ne0$, and set
$$
\tau_*:=\arctan\frac12,\qquad
\theta_*:=\arctan\frac23,
\qquad 0<\tau_*<\theta_*<\frac\pi4.
$$
The continuous map $L$ is positively homogeneous on eight sectors of the plane,
which we denote $S_1,\ldots,S_8$:

\begin{equation}\label{eq:compact-sectors}
\begin{array}{c|c|c|c}
&\text{angular sector}&L(x)&\nabla L\\ \hline
S_1&(0,\tau_*)&0&(0,0)\\
S_2&(\tau_*,\theta_*)&2x_1-4x_2&(2,-4)\\
S_3&(\theta_*,\pi/2)&-x_2&(0,-1)\\
S_4&(\pi/2,3\pi/4)&2x_1-x_2&(2,-1)\\
S_5&(3\pi/4,\pi)&-3x_2&(0,-3)\\
S_6&(\pi,3\pi/2)&0&(0,0)\\
S_7&(3\pi/2,7\pi/4)&-2x_1&(-2,0)\\
S_8&(7\pi/4,2\pi)&2x_2&(0,2)
\end{array}
\end{equation}

The nonlinear locus contains the three pairs of antipodal rays forming the lines 
$$
x_2=0,\qquad x_1=0,\qquad x_1+x_2=0.
$$
and two additional  rays, at angles $\tau_*$ and $\theta_*$, which
lie strictly inside the first quadrant. 

We will see that the three
antipodal pairs constrain the first-layer hyperplanes of a putative representation of $L$; the
two interior rays constrain the second layer; and $L$ vanishing on
sector $S_1$ will make those constraints incompatible.

\begin{lemma}[Convergence]\label{l:compact-convergence}
For $M\ge4$ and $x\in\R^2$,
$$
|f_M(x)-L(x)|\le\frac{4\|x\|}{M-2}.
$$
Consequently $f_M\to L$ in $C_\wsup(\R^2,\R)$, and uniformly
on every compact subset of $\R^2$.
\end{lemma}
\begin{proof}
On the first and third quadrants, $g_0=0$ and hence $f_M=L$.
On the second quadrant, $h_1,h_2\ge0$, so both outer ReLUs are
active and again $f_M=L$.

In the fourth quadrant write $x_1=a>0$ and $x_2=-b<0$.
For $b\ge a$ the outer preactivations are nonnegative when
$M\ge4$, and equality still holds. When $0<b<a$, we have
$$
g_0=b,\qquad h_1=-2a-2b,\qquad h_2=-2a,\qquad L=-2b.
$$
Using $\relu(t)=t+\relu(-t)$ gives
$$
f_M-L
=\relu\bigl(2a-(M-2)b\bigr)
 -\relu(2a-Mb).
$$
The right side lies between $0$ and $2b$, and vanishes when
$b\ge2a/(M-2)$. Otherwise,
$$
0\le f_M-L\le2b<\frac{4a}{M-2}
\le\frac{4\|x\|}{M-2}.
$$
The quadrant boundaries follow by continuity. Dividing by
$1+\|x\|$ gives the weighted-sup assertion.
\end{proof}

\subsection{What the nonlinear locus forces}

We show that $L$ cannot be represented even by a slightly larger
class of homogeneous networks: we allow individual (zero bias) hidden neurons
to apply either ReLU or the identity, keeping the same widths.
Call such a zero-bias network a \emph{local homogeneous
representation}. 
(Allowing identity neurons will be necessary
when we later argue about biased networks on compact domains.)

For each first-layer ReLU neuron with nonzero weight $w_i$, let
$$
H_i=\{x\in\R^2:w_i\cdot x=0\}
$$
be its associated hyperplane. First-layer identity neurons and
zero-weight neurons contribute no associated hyperplane. Write
$\C^{(1)}$ for the polyhedral complex induced by the distinct
associated first-layer hyperplanes, with top-dimensional cells
the closures of the connected components of their complement.

On the interior of each top-dimensional cell of
$\C^{(1)}$, every second-layer preactivation is a
homogeneous linear function. If at least one
first-layer hyperplane is present, each such cell
lies in a half-plane bounded by a line through the
origin. Consequently, the zero line of a nonzero
second-layer preactivation can contribute at most
one nonlinear ray to the interior of that cell.

If no first-layer hyperplane is present, however,
the entire plane is a single first-layer cell,
and a second-layer ReLU may contribute two
antipodal nonlinear rays.

The antipodal pairs of rays give a particularly
strong restriction when the first layer consists entirely of
ReLUs.

\begin{lemma}[Antipodal rays and associated hyperplanes]
\label{l:antipodal}
Let $F:\R^2\to\R$ have a zero-bias representation with at most
three first-layer neurons, each applying ReLU; neurons in later hidden layers may apply ReLU or the identity.   If $x\ne0$ and both $x$ and $-x$ belong
to the nonlinear locus $\N(F)$, then $x$ lies in a hyperplane associated to a
nonzero first-layer ReLU neuron.
\end{lemma}

\begin{proof}
Suppose $x$ does not lie in a hyperplane associated to a nonzero first-layer ReLU neuron.  Then $x$ and $-x$ lie in the relative interiors of
antipodal top-dimensional cells of $\C^{(1)}$. 
Each first-layer
ReLU with nonzero weight is active on precisely one of these
two cells, while a ReLU with zero weight outputs $\relu(0)=0$
identically and is active on neither; since there are at most
three first-layer units, at least one cell has at most one
active neuron.
On that cell the first-layer output is zero or a positive
scalar multiple of a fixed coordinate vector. The subsequent
zero-bias network is positively homogeneous, and therefore
its restriction to that scalar ray is linear. Thus $F$ is
linear near at least one of $x,-x$, a contradiction.
\end{proof}

\begin{lemma}[The first quadrant cannot be a single first-layer cell]
\label{l:inactive-wedge}
Let $F:\R^2\to\R$ be given by a zero-bias network whose two
second-layer neurons both apply ReLU, and suppose both
second-layer preactivations are linear on the open first
quadrant $Q=\{x_1>0,\ x_2>0\}$. Suppose further that $F$ has
two distinct nonlinear rays in $Q$, and that $F$ vanishes on
the open sector $V$ between the positive $x_1$-axis and the
nonlinear ray of smaller angle. Then $F$ vanishes on a
neighborhood of each point $(a,0)$ with $a>0$.
\end{lemma}
\begin{proof}
Write $F=c_1\relu(G_1)+c_2\relu(G_2)$, where $G_1,G_2$ are the
second-layer preactivations, and let $\lambda_1,\lambda_2$ be the
linear forms with $G_j=\lambda_j$ on $Q$. On $Q$, the function
$F$ can fail to be linear only along the zero line of some
$\lambda_j$ with $c_j\ne0$ and $\lambda_j\ne0$, and each such
zero line meets $Q$ in at most one ray. Hence the two nonlinear
rays lie on the zero lines of different forms, so
$c_1,c_2\ne0$, $\lambda_1,\lambda_2\ne0$, and the two zero lines
are distinct. Moreover, the zero line of each $\lambda_j$ meets
$Q$ only in one of these two rays, neither of which meets $V$.
Thus each $\lambda_j$ has a constant strict sign on $V$.

We claim both forms are negative on $V$. If exactly one,
say $\lambda_j$, were positive there, then
$F=c_j\lambda_j$ on $V$ would be a nonzero linear form,
contradicting $F|_V=0$. If both were positive, then
$c_1\lambda_1+c_2\lambda_2=0$ on the open set $V$, hence
everywhere, so $\lambda_1$ and $\lambda_2$ would be nonzero
multiples of each other and their zero lines would coincide,
a contradiction. This proves the claim.

Now fix $a>0$. The point $(a,0)$ lies in the closure of $V$,
so $\lambda_j(a,0)\le0$ by continuity. The zero line of
$\lambda_j$ passes through a nonlinear ray in $Q$, so it is not
the horizontal axis, and therefore $\lambda_j(a,0)\ne0$. Hence
$\lambda_j(a,0)<0$. Each $G_j$ is continuous on $\R^2$ and
agrees with $\lambda_j$ on $Q$, so
$G_j(a,0)=\lambda_j(a,0)<0$, and by continuity both $G_j$ are
negative on a neighborhood of $(a,0)$. Both ReLUs vanish
there, and so does $F$.
\end{proof}

\begin{proposition}[Three first-layer ReLUs are impossible]
\label{p:three-first-layer}
No local homogeneous representation of $L$ can have all three
first-layer units applying ReLU.
\end{proposition}
\begin{proof}
The nonlinear locus contains three distinct antipodal pairs,
on $x_2=0$, $x_1=0$, and $x_1+x_2=0$. By
Lemma~\ref{l:antipodal}, these three lines must be the
hyperplanes associated to the three first-layer ReLUs. Hence
the first quadrant is contained in the interior of one
top-dimensional cell of $\C^{(1)}$, and each second-layer
preactivation is linear there. Both second-layer neurons must
apply ReLU, since the two nonlinear rays in the first quadrant
cannot be produced by a single ReLU or by identity neurons.
Lemma~\ref{l:inactive-wedge} now says the network is zero
near every point $(a,0)$, $a>0$. This contradicts $L=2x_2$
on $S_8$, immediately below the horizontal axis.
\end{proof}

It remains possible that a first-layer neuron is linear.
In that case the antipodal argument does not force all three
associated hyperplanes, since a linear neuron contributes no
hyperplane. We next show that the nonlinear rays still leave
only two possible first-layer arrangements, and that the
gradients of $L$ exclude both.

\subsection{The remaining first-layer complexes}

\begin{lemma}[The necessary neuron types]
\label{l:layer-types}
Any local homogeneous representation of $L$ must have two
second-layer ReLU neurons and at least two distinct
hyperplanes associated to first-layer ReLU neurons.
\end{lemma}
\begin{proof}
If neither second-layer unit applies ReLU, the nonlinear locus
is contained in at most three first-layer hyperplanes and has
at most six rays. If there is no associated first-layer
hyperplane, then even two second-layer ReLUs can introduce
at most four rays. With just one associated first-layer
hyperplane, the two first-layer cells contain at most two
interior nonlinear rays each; together with the two rays on
the hyperplane this gives at most six.

Suppose just one second-layer unit applies ReLU. With two
distinct first-layer hyperplanes there can be at most four
rays on these hyperplanes and at most one additional ray in
each of the four top-dimensional cells. To obtain the eight
rays of $L$, all four rays on the two hyperplanes would have
to be nonlinear. Therefore the two associated hyperplanes
would be among $x_2=0$, $x_1=0$, and $x_1+x_2=0$, none of
which meets the open first quadrant. But the two nonlinear
rays $\tau_*,\theta_*$ inside that quadrant cannot both
come from one second-layer ReLU. If there are three distinct
first-layer hyperplanes, all three first-layer neurons are
ReLUs, which Proposition~\ref{p:three-first-layer} already
excludes. The cases with at most one associated hyperplane
were counted above. This proves the assertion.
\end{proof}

Combined with Proposition~\ref{p:three-first-layer}, the lemma
shows that any putative representation of $L$ must have
exactly two nondegenerate first-layer ReLU neurons with
distinct associated hyperplanes. The remaining first-layer
neuron may be taken to be an identity neuron, possibly with
zero incoming weights. Thus we can write the second-layer
preactivations as
\begin{equation}\label{eq:two-hinges-linear}
G_j(x)=\alpha_j\relu(w_1\cdot x)
+\beta_j\relu(w_2\cdot x)+\gamma_j(a\cdot x),
\qquad j=1,2,
\end{equation}
where $w_1^\perp$ and $w_2^\perp$ are the two distinct
associated hyperplanes and $a$ may be zero.

There is one further restriction, using gradients rather than
just the number of nonlinear rays in a cell.

\begin{lemma}[Two-ray gradient constraint]\label{l:two-ray}
Let $F:\R^2\to\R$ be given by a zero-bias network whose two
second-layer neurons both apply ReLU, and let $C$ be a
top-dimensional first-layer cell contained in an open
half-plane. Suppose $F$ has two distinct nonlinear rays
$r_1,r_2$ in the interior of $C$, with $r_2$ following $r_1$
in the counterclockwise direction. Let $p$ be the gradient of
$F$ on the open sector between $r_1$ and $r_2$, and for
$i=1,2$ let $\Delta_i$ be the gradient jump of $F$ across $r_i$,
taken in the counterclockwise direction. Then
\begin{equation}\label{eq:two-ray-obstruction}
p\in\{0,\ \Delta_1,\ -\Delta_2,\ \Delta_1-\Delta_2\}.
\end{equation}
In particular, in a local homogeneous representation of $L$,
neither the pair of rays at $\theta_*,\pi/2$ nor the pair at
$\pi/2,3\pi/4$ can lie in the interior of a single
first-layer cell.
\end{lemma}
\begin{proof}
Write $F=c_1\relu(G_1)+c_2\relu(G_2)$. On $C$ each $G_j$ agrees
with a linear form, so in the interior of $C$ the function $F$
can fail to be linear only along the zero line of some $G_j$
with $c_j\ne0$ and $G_j\not\equiv0$ on $C$. Since $C$ lies in
an open half-plane, each zero line meets the interior of $C$ in
at most one ray. Hence, after relabeling, $r_j$ lies on the
zero line of $G_j$, and $r_1,r_2$ are the only nonlinear rays of
$F$ in the interior of $C$. In particular $F$ is linear on the
open sector $W$ between $r_1$ and $r_2$, so $p$ is defined.

The zero line of $G_1$ meets the interior of $C$ only in $r_1$,
so $G_1$ has a constant strict sign on each side of $r_1$
in $C$, and these two signs are opposite. Since $G_2$ does not
change sign across $r_1$, the jump $\Delta_1$ comes from the
first neuron alone: if that neuron is active on $W$, then
$\Delta_1=c_1\nabla G_1$, which is its contribution to $p$;
otherwise its contribution to $p$ is $0$. Similarly, $G_2$
changes sign across $r_2$ while $G_1$ does not, so if the
second neuron is active on $W$, then $\Delta_2=-c_2\nabla G_2$
and its contribution to $p$ is $-\Delta_2$; otherwise its
contribution is $0$. Adding the two contributions proves
\eqref{eq:two-ray-obstruction}.

For the final assertion, Lemma~\ref{l:layer-types} shows that
any local homogeneous representation of $L$ has two
second-layer ReLUs and at least two distinct associated
first-layer hyperplanes, so every top-dimensional first-layer
cell lies in an open half-plane. For the rays at
$\theta_*,\pi/2$, the data in \eqref{eq:compact-sectors} give
$$
p=(0,-1),\qquad \Delta_1=(-2,3),\qquad \Delta_2=(2,0),
$$
and the set in \eqref{eq:two-ray-obstruction} is
$\{(0,0),(-2,3),(-2,0),(-4,3)\}$. For the rays at
$\pi/2,3\pi/4$,
$$
p=(2,-1),\qquad \Delta_1=(2,0),\qquad \Delta_2=(-2,-2),
$$
and the set is $\{(0,0),(2,0),(2,2),(0,-2)\}$. In neither case
does $p$ belong to the set.
\end{proof}

The locations of the eight nonlinear rays now restrict the
possible positions of these two hyperplanes. To express their
positions, denote by $\psi$ the angle of the line meeting the
open first quadrant, and by $\rho$ the angle of the other
line, when this configuration applies. Each angle is taken
modulo $\pi$, with $\rho=\pi$ representing the horizontal
line.

\begin{lemma}[Possible first-layer hyperplane arrangements]
\label{l:two-hyperplane-arrangements}
The associated first-layer hyperplanes must have angles
\begin{equation}\label{eq:two-possible-arrangements}
\rho=\pi/2,\quad0<\psi\le\tau_*
\qquad\text{or}\qquad
\rho=3\pi/4,\quad\psi=\theta_*.
\end{equation}
\end{lemma}
\begin{proof}
If neither associated hyperplane meets the open first quadrant,
then both second-layer preactivations are linear there, and
Lemma~\ref{l:inactive-wedge} contradicts the values of $L$
across the positive horizontal axis. If both meet that
quadrant, at angles $\psi_1<\psi_2$, the first-layer cell
between $\psi_2$ and $\psi_1+\pi$ contains three nonlinear
rays, at $\pi/2,3\pi/4,\pi$, although only two second-layer
ReLUs are available. Thus precisely one associated hyperplane
meets $Q$; write its angle as $\psi\in(0,\pi/2)$, and the
other as $\rho\in[\pi/2,\pi]$.

The four top-dimensional cells of $\C^{(1)}$, in cyclic order,
have interiors given by the angle intervals
$$
C_1:(\rho-\pi,\psi),\quad
C_2:(\psi,\rho),\quad
C_3:(\rho,\psi+\pi),\quad
C_4:(\psi+\pi,\rho+\pi),
$$
interpreted modulo $2\pi$. If $\rho=\pi/2$, then $C_1$
contains the nonlinear rays $7\pi/4$ and $0$, so it cannot
also contain $\tau_*$. Consequently $\psi\le\tau_*$.

Suppose $\rho>\pi/2$. If $\psi<\theta_*$, the rays
$\theta_*,\pi/2$ lie inside $C_2$. If $\rho>3\pi/4$,
the rays $\pi/2,3\pi/4$ also lie inside $C_2$. Both possibilities are excluded by Lemma~\ref{l:two-ray}.
We conclude that $\psi\ge\theta_*$ and $\rho\le3\pi/4$.

If $\rho<3\pi/4$, then $C_1$ contains the three nonlinear
rays $7\pi/4,0,\tau_*$, impossible with two second-layer
ReLUs. Hence $\rho=3\pi/4$. If $\psi>\theta_*$, the same
cell contains $0,\tau_*,\theta_*$, again impossible. Thus
$\psi=\theta_*$.
\end{proof}

We have reduced all possible first-layer canonical complexes
to two configurations. The locations of the nonlinear rays
alone cannot eliminate these last cases; here the \emph{affine
gradients} of $L$ are decisive. Across a codimension-one face
contained in a hyperplane $H=\{w\cdot x=0\}$ associated to a
first-layer neuron, the gradient of each second-layer
preactivation can change only by a multiple of $w$. Indeed,
this is immediate from \eqref{eq:two-hinges-linear}:
all first-layer features other than the one associated to $H$
have the same linear formula on both sides. This is also an
instance of the normality of gradient jumps in
Lemma~\ref{l:jump-normal}.

\begin{lemma}[The two remaining arrangements are impossible]
\label{l:two-arrangements-impossible}
Neither configuration in \eqref{eq:two-possible-arrangements}
is compatible with the gradients in \eqref{eq:compact-sectors}.
\end{lemma}

\begin{proof}
By Lemma~\ref{l:layer-types}, both second-layer neurons apply
ReLU, so $L=c_1\relu(G_1)+c_2\relu(G_2)$ with $G_j$ as in
\eqref{eq:two-hinges-linear}. Every top-dimensional first-layer
cell lies in an open half-plane, so on each such cell the
zero line of each $G_j$ meets its interior in at most one ray.
We treat the two first-layer arrangements separately.

\emph{First configuration: $\rho=\pi/2$, $0<\psi\le\tau_*$.}
Label the first-layer ReLUs in \eqref{eq:two-hinges-linear} so
that $w_1^\perp=\{x_1=0\}$ and $w_2^\perp$ is the associated
hyperplane at angle $\psi$. The interior of $C_1$ contains
exactly two nonlinear rays, at $7\pi/4$ and $0$, together with
$S_7$, $S_8$, and the zero sector $(0,\psi)\subseteq S_1$.
Likewise, the interior of $C_3$ contains exactly the nonlinear
rays $3\pi/4$ and $\pi$, together with $S_4$, $S_5$, and the
zero sector $(\pi,\pi+\psi)\subseteq S_6$.

Consider $C_1$. The two rays lie on the zero lines in $C_1$ of
different preactivations, so these zero lines are distinct and
both $c_j\ne0$. Both neurons are inactive on $(0,\psi)$: if
exactly one were active, $L$ would equal a nonzero linear form
there, and if both were active, $c_1G_1+c_2G_2=0$ on an open
set would force their zero lines to coincide. Crossing the ray
at $0$ clockwise into $S_8$ activates one neuron, whose
contribution there is $L=2x_2<0$; hence its output coefficient
is negative. Its zero line meets $C_1$ only in the ray at $0$,
so it remains active on $S_7$. Crossing the ray at $7\pi/4$
into $S_7$ activates the other neuron, whose contribution there
is $-2x_1-(2x_2)=-2(x_1+x_2)>0$; hence its output coefficient
is positive. After positive rescaling and relabeling, we may
therefore write $L=-\relu(G_-)+\relu(G_+)$, and reading off
the active contributions on $S_8$ and $S_7$ gives
$G_-=-2x_2$ and $G_+=-2(x_1+x_2)$ on $C_1$.

The same reasoning applies on $C_3$. Both neurons are inactive
on $(\pi,\pi+\psi)$. Crossing the ray at $\pi$ clockwise into
$S_5$ activates a neuron whose contribution there is
$L=-3x_2<0$; since its output coefficient is negative, this
neuron is the one with preactivation $G_-$. It remains active
on $S_4$, and crossing the ray at $3\pi/4$ activates the other
neuron, $G_+$, whose contribution on $S_4$ is
$(2x_1-x_2)-(-3x_2)=2(x_1+x_2)>0$. In summary,
\begin{equation}\label{eq:two-anchor-pairs}
\begin{array}{c|cc}
&G_-&G_+\\\hline
C_1&-2x_2&-2(x_1+x_2)\\
C_3&3x_2&2(x_1+x_2)
\end{array}
\end{equation}

Exactly one of the four first-layer cells has both first-layer
ReLUs inactive. On that cell $G_\pm=\gamma_\pm(a\cdot x)$, so
$L$ there is a function of $a\cdot x$ alone. Consequently its
gradients on all sectors of that cell are parallel to $a$, and
its interior contains at most one nonlinear ray, namely one
lying on $a^\perp$. This excludes $C_1$ and $C_3$, each of
which contains two nonlinear rays. It also excludes $C_2$:
if $\psi<\tau_*$, then $C_2$ contains the two nonlinear rays
$\tau_*$ and $\theta_*$, and if $\psi=\tau_*$, then $C_2$
contains $S_2$ and $S_3$, whose gradients $(2,-4)$ and
$(0,-1)$ are not parallel. Hence the inactive cell is $C_4$.

The cell $C_4$ is adjacent to $C_1$ across the ray at $3\pi/2$,
which lies in $w_1^\perp$, and to $C_3$ across the ray at
$\psi+\pi$, which lies in $w_2^\perp$. Crossing an associated
hyperplane changes the state of only the corresponding ReLU.
Therefore $\relu(w_1\cdot x)$ is active and
$\relu(w_2\cdot x)$ inactive on $C_1$, while
$\relu(w_2\cdot x)$ is active and $\relu(w_1\cdot x)$
inactive on $C_3$.

On $C_1$ we thus have
$G_\pm=\alpha_\pm(w_1\cdot x)+\gamma_\pm(a\cdot x)$, and
$w_1\cdot x$ is a multiple of $x_1$. Comparing $x_2$-coefficients
with \eqref{eq:two-anchor-pairs} gives
$\gamma_-a_2=\gamma_+a_2=-2$, so $a_2\ne0$ and
$\gamma_-=\gamma_+$. Hence the identity terms cancel in
$G_--G_+$, and on $C_3$, where only $\relu(w_2\cdot x)$ is
active,
$$
G_--G_+=(\beta_--\beta_+)\,(w_2\cdot x).
$$
By \eqref{eq:two-anchor-pairs}, the left side equals
$3x_2-2(x_1+x_2)=-2x_1+x_2\ne0$ on $C_3$. Thus $w_2$ is a
nonzero multiple of $(-2,1)$, and the associated hyperplane
$w_2^\perp$ is the line $x_2=2x_1$, at angle
$\arctan 2>\tau_*$. This contradicts $\psi\le\tau_*$.

\emph{Second configuration: $\rho=3\pi/4$, $\psi=\theta_*$.}
Now $C_1$ is the open sector from $-\pi/4$ to $\theta_*$. It
contains the two nonlinear rays $0$ and $\tau_*$, which separate
the zero sector $S_1$ from $S_8$ and $S_2$, respectively. The
two second-layer preactivations are linear on $C_1$, and the two
rays lie on the zero lines in $C_1$ of different preactivations,
so these zero lines are distinct. As in the first configuration,
both neurons are therefore inactive on $S_1$, since their nonzero
contributions could not cancel on an open sector. Crossing the
ray at $0$ from $S_1$ into $S_8$ activates one neuron, whose
contribution there is $L=2x_2<0$; crossing the ray at $\tau_*$
from $S_1$ into $S_2$ activates the other, whose contribution
there is $L=2x_1-4x_2<0$. Thus both output coefficients are
negative. After positive rescaling and relabeling, we may write
$L=-\relu(G_1)-\relu(G_2)$, where
\begin{equation}\label{eq:first-cell-gradients}
G_1=-2x_2,\qquad G_2=-2x_1+4x_2
\qquad\text{on }C_1.
\end{equation}

In the neighboring cell $C_2$, from angle $\theta_*$ to
$3\pi/4$, the only interior nonlinear ray is $\pi/2$.
Here
$$
-L=\begin{cases}
x_2,&\theta_*<\tau<\pi/2,\\
x_2-2x_1,&\pi/2<\tau<3\pi/4,
\end{cases}
$$
and $-L>0$ throughout $C_2$, since $x_2>0$ there and $x_1<0$
when $\tau>\pi/2$.

On $C_2$ both preactivations are linear and
$-L=\relu(G_1)+\relu(G_2)$. Since the ray at $\pi/2$ is
nonlinear and lies in the interior of $C_2$, one
preactivation, call it $G_a$, is a nonzero multiple of $x_1$
there; let $G_b$ be the other, so that
$\{G_a,G_b\}=\{G_1,G_2\}$. We claim that $G_b>0$
throughout $C_2$. First, $G_b$ is not a multiple of $x_1$
(possibly zero): otherwise $-L$ would depend only on $x_1$
on $C_2$, whereas $-L=x_2$ for $\theta_*<\tau<\pi/2$.
Second, $G_b$ does not vanish on any other ray $r$ in the
interior of $C_2$: since $C_2$ lies in an open half-plane,
the line $x_1=0$ meets it only in the ray at $\pi/2$, so
$G_a\ne0$ on $r$, and across $r$ the gradient of $-L$
would jump by $\pm\nabla G_b\ne0$, making $r$ a second
nonlinear ray in $C_2$. Hence $G_b$ has a constant strict
sign on $C_2$. It cannot be negative, for then
$-L=\relu(G_a)$ would vanish on one side of $x_1=0$,
contradicting $-L>0$ on $C_2$. This proves the claim.

If $G_a$ were active on $x_1>0$, its contribution there
would be $2x_1$, forcing $G_b=x_2-2x_1$, which is negative
near the ray at $\theta_*$. Thus $G_a$ is active on $x_1<0$,
contributes $-2x_1$ there, and $G_b=x_2$. Hence $G_a=-2x_1$
and $G_b=x_2$ on $C_2$; in terms of the original labels,
$$
\{G_1,G_2\}=\{x_2,-2x_1\}\qquad\text{on }C_2.
$$

The cells $C_1$ and $C_2$ meet along the ray at angle
$\psi=\theta_*$, which lies in the hyperplane associated to
a first-layer ReLU, with normal $(-2,3)$. By
\eqref{eq:two-hinges-linear}, only that ReLU changes state
across this ray, so the gradient of each $G_j$ can change
between $C_1$ and $C_2$ only by a multiple of $(-2,3)$. But
$G_1=-2x_2$ on $C_1$ by \eqref{eq:first-cell-gradients},
while on $C_2$ it equals either $x_2$ or $-2x_1$. The
corresponding gradient changes are $(0,3)$ and $(-2,2)$,
neither of which is parallel to $(-2,3)$. This is the
required contradiction.
\end{proof}

\begin{proposition}[Nonrepresentability of the limit]
\label{p:no-local-rep}
The function $L$ in \eqref{eq:compact-limit} has no local
homogeneous representation of architecture $(2,3,2\mid1)$.
\end{proposition}
\begin{proof}
Proposition~\ref{p:three-first-layer} excludes a first layer
of three ReLUs. Lemma~\ref{l:layer-types} forces two genuine
second-layer ReLUs and at least two distinct first-layer
associated hyperplanes. Thus any remaining representation has
two nondegenerate first-layer ReLUs and one identity unit.
Lemma~\ref{l:two-hyperplane-arrangements} leaves only the two
arrangements in \eqref{eq:two-possible-arrangements}, and
Lemma~\ref{l:two-arrangements-impossible} excludes both.
\end{proof}

\subsection{From the local obstruction to compact-domain non-closedness}

The preceding proposition rules out more than a global bias-free
realization: it rules out every homogeneous realization obtained
by allowing some hidden neurons to act linearly. This is exactly
what is needed for a compact domain, since a biased neuron can
be locally active throughout a neighborhood and hence act as an
identity there.

\begin{lemma}\label{l:localization}
Let $F\in\F_A$ be a ReLU network with biases and affine output
layer, and let $p$ lie in its input space. There exist $\delta>0$
and a positively homogeneous network $G$ with the same widths,
zero biases, and hidden units applying either ReLU or the
identity, such that
$$
F(p+x)=F(p)+G(x)\qquad(\|x\|<\delta).
$$
Consequently, if $L$ is positively homogeneous and
$F(p+x)=F(p)+L(x)$ near zero, then $G=L$ globally.
\end{lemma}
\begin{proof}
Translate the input to assume $p=0$. Inductively suppose the
previous layer is $e+H(x)$ near zero, where $H(0)=0$ and $H$
is positively homogeneous. The next preactivation is
$\beta+WH(x)$. On a sufficiently small neighborhood, the
coordinates with $\beta_i>0$ remain positive, and those with
$\beta_i<0$ remain negative. After applying ReLU and
subtracting the value at zero, their variations are respectively
$(WH(x))_i$ and $0$; coordinates with $\beta_i=0$ have
variation $\relu((WH(x))_i)$. A zero variation is realized by
a ReLU unit with zero incoming weights. Thus at each layer
the variation is given by a zero-bias network of the same
widths using ReLU and identity units. The affine output layer
completes the induction. If $G=L$ near zero, their positive
homogeneity gives equality everywhere by scaling.
\end{proof}

\begin{theorem}[Non-closedness for $(2,3,2\mid1)$; Theorem~\ref{mt:nonclosed}]
\label{t:nonclosed}
Let $A=(2,3,2\mid1)$. For every compact
$\Omega\subset\R^2$ with nonempty interior, the realization
space $\F_A|_\Omega$ is not closed in $C(\Omega)$ with the
uniform norm. Moreover, neither $\F_A$ nor $\F_A^0$ is
closed in $C_\wsup(\R^2,\R)$.
\end{theorem}
\begin{proof}
Lemma~\ref{l:compact-convergence} gives $f_M\to L$ in
weighted sup norm, with $f_M\in\F_A^0$. If $L\in\F_A$,
localization at the origin would give a local homogeneous
representation of $L$, contradicting
Proposition~\ref{p:no-local-rep}. Thus neither full-space
realization space is closed.

For a compact $\Omega$ with nonempty interior, choose
$p\in\operatorname{int}\Omega$. The translated functions
$f_M(\,\cdot-p)$ belong to $\F_A$ because biases are
allowed and converge uniformly on $\Omega$ to
$L(\,\cdot-p)$. If this limit were realized on $\Omega$,
translating a representing network back to the origin would
give a network agreeing with $L$ near zero. Lemma~\ref{l:localization}
and Proposition~\ref{p:no-local-rep} yield a contradiction.
\end{proof}

\section{Blow-ups and residues in limits of ReLU networks}
\label{s:blowup-residues}

The counterexample in Section~\ref{s:nonclosed} has a feature that
is not specific to its architecture: some hidden features converge
to the same function, while increasingly large output coefficients
magnify the difference between them. The limiting features cancel,
but their magnified difference need not vanish. We call the
surviving contribution a \emph{blow-up residue}.

We show that this is the only possible source of non-closedness in
the weighted supremum norm. The result is a necessary condition,
not a criterion for non-closedness: a residue can also be
realizable by a different choice of network parameters.

The behavior of degenerating parameter sequences is also
relevant to the existence of optimal network approximations.
Dereich and Kassing~\cite{dereich2024}, and Dereich,
Jentzen, and Kassing~\cite{dereichjentzenkassing2024},
study existence of minimizers for shallow ReLU network
classes by enlarging the class of network responses to
include suitable limiting objects. Our perspective is
different: rather than studying minimizers for a
particular loss, we seek to identify the possible limits
of functions realized by a fixed architecture. The
blow-up residues introduced below describe one mechanism
by which such limits can arise from unbounded parameter
sequences.

\subsection{Normalization and blow-up residues}

\begin{lemma}[Compactness of normalized hidden layers]
\label{l:normalized-hidden-layers}
Let $A=(n_0,n_1,\ldots,n_d\mid1)$ be a fixed ReLU architecture,
with $d\ge1$ hidden layers and an affine scalar output layer.
For any sequence $f_m\in\F_A$, there are representations of the
$f_m$ and a subsequence such that all hidden-layer weights and
biases converge. In particular, if $N=n_d$, the representations
can be written
\[
 f_m(x)=b_m+\sum_{j=1}^N a_{m,j}h_{m,j}(x),
\]
where $h_{m,j}$ is the $j$th coordinate of the final hidden-layer
map and, for each $j$, there is a function
$h_j\in C_\wsup(\R^{n_0},\R)$ such that
\[
 h_{m,j}\longrightarrow h_j
 \qquad\text{in }C_\wsup(\R^{n_0},\R).
\]
No boundedness of the final output coefficients $b_m,a_{m,j}$
is asserted.
\end{lemma}

\begin{proof}
For each hidden neuron, consider the vector consisting of its
incoming weights and bias. If this vector has norm $s>0$, divide
it by $s$ and multiply all outgoing weights of that neuron by $s$.
Positive homogeneity of ReLU leaves the realized function unchanged.
If the incoming vector is zero, leave it unchanged. Applying these
rescalings successively through the hidden layers gives a
representation in which every hidden incoming weight-and-bias
vector has norm at most one. Since the architecture is fixed, the
hidden parameters now lie in a compact subset of a finite-dimensional
space. Pass to a subsequence on which they all converge.

To verify convergence of the hidden features in the weighted
supremum norm, let $z_m^{(k)}$ and $z^{(k)}$ be the $k$th hidden-layer
maps for the approximating and limiting hidden parameters, with
$z_m^{(0)}(x)=z^{(0)}(x)=x$. If $W_{m,k}\to W_k$ and
$b_{m,k}\to b_k$, the $1$-Lipschitz property of coordinatewise
ReLU gives
\begin{align*}
\|z_m^{(k)}-z^{(k)}\|_\wsup
&\le \|W_{m,k}-W_k\|\,\|z_m^{(k-1)}\|_\wsup\\
&\quad+\|W_k\|\,\|z_m^{(k-1)}-z^{(k-1)}\|_\wsup
       +\|b_{m,k}-b_k\|.
\end{align*}
(The weighted supremum norm is applied componentwise using the
Euclidean norm.) Induction on $k$ shows that the preceding-layer
norms stay bounded and that $z_m^{(k)}\to z^{(k)}$ in this norm.
Taking coordinates of the final hidden-layer maps proves the claim.
\end{proof}

\begin{definition}[Blow-up residue]
\label{def:blowup-residue}
Let $X$ be a Banach space and let
$U_m,U:\R^K\to X$ be bounded linear maps with
$U_m\to U$ in operator norm. Suppose $U_m\alpha_m\to f$ in $X$.
A \emph{blow-up decomposition} of the coefficient sequence,
after passing to a subsequence, is an expression
\[
 \alpha_m=\beta_m+M_m\gamma_m,
 \qquad
 \beta_m\to\beta,\quad M_m\to\infty,\quad
 \|\gamma_m\|=1,\quad \gamma_m\to\gamma\in\ker U.
\]
If the limit
\[
 R=\lim_{m\to\infty}M_mU_m\gamma_m
\]
exists in $X$, we call $R$ the \emph{blow-up residue} of this
decomposition. Neither the decomposition nor its residue is
claimed to be canonical.
\end{definition}

\begin{example}[A simple blow-up residue]
\label{ex:elementary-residue}
A first illustration is the pair of nearly coincident hinges
$\relu(x+1/M)$ and $\relu(x)$. Their magnified difference
\[
 M\bigl(\relu(x+1/M)-\relu(x)\bigr)
\]
converges pointwise to a step function. This does not give a
weighted-supremum limit, however, since that pointwise limit is
discontinuous.

Here is a closely related example that does converge in the
weighted supremum norm. Let $u(x)=\relu(x)$ and
$v(x)=\relu(-x)$, and, for $M\ge1$, put
\[
 h_{M,1}(x)=\relu\left(u(x)+\frac{v(x)}{M}\right),
 \qquad h_{M,2}(x)=\relu(u(x)).
\]
Then $h_{M,1},h_{M,2}\to u$ in weighted supremum norm, but
\begin{equation}\label{eq:elementary-residue}
 M(h_{M,1}-h_{M,2})=v
 \qquad\text{for every }M\ge1.
\end{equation}
In particular, $v$ is a nonzero blow-up residue outside
$\operatorname{span}\{1,u\}$. Nevertheless, $v$ is realizable by
the same architecture $(1,2,2\mid1)$.
\end{example}

\begin{proof}
The first expression tends to $0$ for $x<0$ and to $1$
for $x\ge0$, so its pointwise limit is discontinuous. For the
second construction, $u,v\ge0$ implies
\[
 h_{M,1}=u+M^{-1}v,\qquad h_{M,2}=u.
\]
This proves both weighted-supremum convergence of the features
and \eqref{eq:elementary-residue}. The corresponding maps
\[
 U_M(a_0,a_1,a_2)=a_0+a_1h_{M,1}+a_2h_{M,2}
\]
converge in operator norm to
$U(a_0,a_1,a_2)=a_0+(a_1+a_2)u$.
The output coefficient vector $(0,M,-M)$ diverges in the
direction $(0,1,-1)\in\ker U$, and the associated residue is
$v$. Since $u=0$ on $(-\infty,0)$ while $v$ is nonconstant
there, $v\notin\operatorname{span}\{1,u\}$.
Finally, the first hidden layer $\bigl(u(x),v(x)\bigr)$ and a
second hidden layer that selects $v$ realize $v$ with widths
$(2,2)$.
\end{proof}

\begin{proposition}[Blow-up residue decomposition]
\label{p:blowup-residue-principle}
Let $X$ be a Banach space, let $K\ge1$, and let
$U_m,U:\R^K\to X$ be bounded linear maps such that
$U_m\to U$ in operator norm. Suppose
$\alpha_m\in\R^K$ and
\[
 U_m\alpha_m\longrightarrow f\in X.
\]
Set $E=\operatorname{im}U$. If $f\notin E$, then, after passing
to a subsequence, $\alpha_m$ has a blow-up decomposition
\[
 \alpha_m=\beta_m+M_m\gamma_m,
 \qquad
 \beta_m\to\beta,\quad M_m\to\infty,\quad
 \|\gamma_m\|=1,\quad \gamma_m\to\gamma\in\ker U,
\]
for which the blow-up residue exists and satisfies
\begin{equation}
\label{eq:blowup-residue-decomposition}
 f=U\beta+R,
 \qquad
 R=\lim_{m\to\infty}M_mU_m\gamma_m\notin E.
\end{equation}
In particular, $R\ne0$ and $\ker U\ne\{0\}$.
\end{proposition}

\begin{proof}
Since $E$ is finite-dimensional, there is a bounded linear
projection $P:X\to E$. Set $V=(\ker U)^\perp\subset\R^K$.
The restriction $U|_V:V\to E$ is an isomorphism. Because
$P U_m|_V\to U|_V$, the maps $P U_m|_V$ are isomorphisms for
all sufficiently large $m$, and their inverses converge to
$(U|_V)^{-1}$.

For such $m$, define
\[
 \beta_m=(P U_m|_V)^{-1}(P U_m\alpha_m),
 \qquad \kappa_m=\alpha_m-\beta_m.
\]
The hypothesis $U_m\alpha_m\to f$ implies
\[
 \beta_m\longrightarrow
 \beta:=(U|_V)^{-1}(Pf).
\]
Moreover, $P U_m\kappa_m=0$, and hence
\[
 U_m\kappa_m
 =U_m\alpha_m-U_m\beta_m
 \longrightarrow f-U\beta=:R.
\]
Since $P R=0$ and $f\notin E$, we have $R\ne0$ and
$R\notin E$.

The vectors $\kappa_m$ cannot have a bounded subsequence:
otherwise a further subsequence would converge to $\kappa$,
forcing $R=U\kappa\in E$. Thus $\|\kappa_m\|\to\infty$.
Set
\[
 M_m=\|\kappa_m\|,
 \qquad \gamma_m=\kappa_m/M_m.
\]
After a subsequence, $\gamma_m\to\gamma$ with
$\|\gamma\|=1$. From $P U_m\gamma_m=0$ and
$P U_m\to U$ we obtain $U\gamma=0$. Finally,
\[
 M_mU_m\gamma_m=U_m\kappa_m\longrightarrow R,
\]
which gives the required decomposition.
\end{proof}

\begin{corollary}[Necessary mechanism for non-closedness; Theorem~\ref{mt:residues}]
\label{cor:residue-necessary}
Let $A=(n_0,n_1,\ldots,n_d\mid1)$, $d\ge1$, be a fixed
ReLU architecture with an affine scalar output layer. Suppose
$f_m\in\F_A$ converges in $C_\wsup(\R^{n_0},\R)$ to a function
$f\notin\F_A$. Then there exist representations of a subsequence
of the $f_m$ with convergent hidden-layer parameters such that,
writing $N=n_d$ and $h_{m,j}\to h_j$ in weighted supremum
norm for their final hidden features, the maps
\begin{align*}
 U_m(\alpha_0,\ldots,\alpha_N)
  &=\alpha_0+\sum_{j=1}^N\alpha_jh_{m,j},\\
 U(\alpha_0,\ldots,\alpha_N)
  &=\alpha_0+\sum_{j=1}^N\alpha_jh_j
\end{align*}
converge in operator norm and admit a blow-up decomposition
of the final output coefficients. The corresponding residue $R$ satisfies
\[
 f=U\beta+R,
 \qquad
 R\notin\operatorname{span}\{1,h_1,\ldots,h_N\}.
\]
In particular, $\ker U\ne\{0\}$; equivalently, the functions
$1,h_1,\ldots,h_N$ are linearly dependent.
\end{corollary}

\begin{proof}
Apply Lemma~\ref{l:normalized-hidden-layers} to obtain
representations with $h_{m,j}\to h_j$ in weighted supremum norm.
This implies $U_m\to U$ in operator norm. Every element of
$E=\operatorname{im}U$ is realizable by architecture $A$, using
the limiting hidden layers and an arbitrary affine output layer.
Since $f\notin\F_A$, we have $f\notin E$. Proposition~\ref{p:blowup-residue-principle}
now gives the decomposition and the stated properties.
\end{proof}

\begin{remark}[What the residue does and does not detect]
\label{r:scope-of-residue}
In Corollary~\ref{cor:residue-necessary}, the space
$E=\operatorname{span}\{1,h_1,\ldots,h_N\}$ is the span of one
limiting collection of hidden features; it is not the entire
realization space $\F_A$. Thus $R\notin E$ need not imply
$R\notin\F_A$, as Example~\ref{ex:elementary-residue}
illustrates. Nor is the residue canonical: other choices of
representations or subsequences may produce different residues.
\end{remark}

\subsection{The counterexample as a blow-up residue}

\begin{proposition}[Residue of the $(2,3,2\mid1)$ sequence]
\label{p:counterexample-as-residue}
Let $\relu(t)=\max\{t,0\}$, and define
\[
 u_1=\relu(x_1),\qquad
 u_2=\relu(x_2),\qquad
 u_3=\relu(x_1+x_2),
\]
\[
 g_0=u_1+u_2-u_3,\qquad
 h_1=-4u_1+u_2+2u_3,\qquad
 h_2=-2u_1+4u_2.
\]
For $M\ge4$, set
\[
 f_M=\relu(Mg_0+h_1)-\relu(Mg_0+h_2),
\]
and let $L$ be the weighted-supremum limit of $f_M$ established
in Lemma~\ref{l:compact-convergence}. Put
\[
 v_{M,1}=\relu\left(g_0+\frac{h_1}{M}\right),
 \qquad
 v_{M,2}=\relu\left(g_0+\frac{h_2}{M}\right).
\]
Then $v_{M,1},v_{M,2}\to g_0$ in
$X=C_\wsup(\R^2,\R)$, and
\[
 L=\lim_{M\to\infty}M(v_{M,1}-v_{M,2})
\]
is a blow-up residue for the limiting feature map
\[
 U(a_0,a_1,a_2)=a_0+(a_1+a_2)g_0.
\]
Moreover, $L\notin\operatorname{im}U
=\operatorname{span}\{1,g_0\}$.
\end{proposition}

\begin{proof}
Since $g_0\ge0$, the $1$-Lipschitz property of ReLU gives
\[
 \|v_{M,j}-g_0\|_\wsup
 \le\frac{\|h_j\|_\wsup}{M}\longrightarrow0
 \qquad(j=1,2).
\]
Positive homogeneity gives
\[
 f_M=M(v_{M,1}-v_{M,2}).
\]
Thus the approximating feature maps
\[
 U_M(a_0,a_1,a_2)
 =a_0+a_1v_{M,1}+a_2v_{M,2}
\]
converge in operator norm to $U$. With
\[
 \alpha_M=(0,M,-M),\quad
 \beta_M=0,\quad \lambda_M=\sqrt2 M,\quad
 \gamma_M=(0,1,-1)/\sqrt2,
\]
we have $\alpha_M=\beta_M+\lambda_M\gamma_M$ and
$\gamma_M\in\ker U$. Lemma~\ref{l:compact-convergence}
therefore shows that the associated residue is exactly $L$.

Finally, $g_0(0)=L(0)=0$, so if $L$ belonged to
$\operatorname{span}\{1,g_0\}$ it would be a scalar multiple
of $g_0$. But $g_0$ vanishes throughout the first quadrant,
whereas $L(1,1)=-1$. Thus $L$ does not belong to this span.
\end{proof}

\begin{remark}[The role of first-layer width]
The function $g_0=\relu(x_1)+\relu(x_2)-\relu(x_1+x_2)$
is nonnegative and vanishes on two full-dimensional regions.
The two second-layer features therefore have the same leading
term $g_0$, but their smaller corrections have different effects
on the regions where $g_0$ vanishes. Unlike the elementary
example above, the resulting residue is not realizable by the
original architecture, as established in Section~\ref{s:nonclosed}.
This explains the role of width three in this construction; it
does not establish that width three is necessary for non-closedness.
\end{remark}


\bibliographystyle{amsplain}
\bibliography{FunctionSpaceGeometryBib}

\end{document}